\documentclass[AMA,STIX1COL]{WileyNJD-v2} 
\usepackage{moreverb,ulem}
\usepackage{tikz}

\usepackage{arydshln}
\usepackage{amsmath}
\usepackage{graphicx} 
\graphicspath{{figures/}} 
\usepackage[absolute]{textpos}

\newcommand{\R}{\mathbb{R}}         
\newcommand{\N}{\mathbb{N}}         
\renewcommand{\S}{\mathbb{S}}          
\renewcommand{\P}{\mathbb{P}}          

\newcommand{\Ac}{\mathcal{A}}
\newcommand{\Bc}{\mathcal{B}}
\newcommand{\Cc}{\mathcal{C}}
\newcommand{\Dc}{\mathcal{D}}

\newcommand{\Gc}{\mathcal{G}}

\newcommand{\Ls}{\mathscr{L}}

\newcommand{\Vb}{\mathbf{V}}
\newcommand{\Pb}{\mathbf{P}}

\newcommand{\cge}{\succcurlyeq}      
\newcommand{\cle}{\preccurlyeq}
\newcommand{\cl}{\prec}
\newcommand{\cg}{\succ}
\newcommand{\Delf}{\mathbf{\Delta}}

\renewcommand{\ss}{\mathrm{ss}}
\newcommand{\diag}{\mathrm{diag}}
\newcommand{\col}{\mathrm{col}}

\newcommand{\tr}{\mathrm{trace}}

\newcommand{\Del}{\Delta}
\newcommand{\del}{\delta}
\newcommand{\ga}{\gamma}
\newcommand{\eps}{\varepsilon}
\newcommand{\la}{\lambda}

\newcommand{\gaoptgs}{\ga_\mathrm{gs}}

\newcommand{\opt}{\mathrm{opt}}

\newcommand{\mat}[2]{\left(\begin{array}{@{}#1@{}}#2\end{array}\right)} 
\newcommand{\smat}[1]{\left(\begin{smallmatrix}#1\end{smallmatrix}\right)}

\newcommand{\teq}[1]{\quad\text{#1}\quad} 
\newenvironment{red_test}{\color{red}}{} 

\renewcommand{\t}{\tilde}
\newcommand{\h}{\hat}

\newcommand{\mstrut}[1]{\rule{0pt}{#1}}
\let\oldhline=\hline 
\renewcommand{\hline}{\oldhline\mstrut{2.5ex}}
\let\oldhdashline=\hdashline 
\renewcommand{\hdashline}{\oldhdashline\mstrut{2.5ex}}

\newtheorem{Algorithm}[]{Algorithm}

\renewenvironment{subequations}{%
	\refstepcounter{equation}%
	\edef\theparentequation{\theequation}%
	\setcounter{parentequation}{\value{equation}}%
	\setcounter{equation}{0}%
	\def\theequation{\theparentequation\alph{equation}}%
	\ignorespaces
}{%
	\setcounter{equation}{\value{parentequation}}%
	\ignorespacesafterend
}

\newcounter{frstequation}
\newcounter{scndequation}
\newcounter{trdequation}

\makeatletter
\newcommand{\scndlabel}[1]{%
	\protected@write \@auxout {}{\string \newlabel{#1}{{\@currentlabel}{\thepage}{}{equation.\thesection.\the\value{parentequation}}{}} }%
	\hypertarget{equation.\thesection.\the\value{parentequation}}{}
}
\makeatother
\makeatletter
\newcommand{\dlabel}[2]{ 
	\setcounter{frstequation}{\value{equation}}%
	\setcounter{scndequation}{\value{equation}}%
	\addtocounter{scndequation}{1}%
	\edef\theequation{\theparentequation\alph{frstequation},\alph{scndequation}}
	\def\@currentlabel{\theparentequation\alph{frstequation}}\scndlabel{#1}
	\def\@currentlabel{\theparentequation\alph{scndequation}}\scndlabel{#2}
	\addtocounter{equation}{1}
}
\makeatother

\makeatletter
\newcommand{\tlabel}[3]{ 
	\setcounter{frstequation}{\value{equation}}
	\setcounter{scndequation}{\value{equation}}%
	\setcounter{trdequation}{\value{equation}}%
	\addtocounter{scndequation}{1}%
	\addtocounter{trdequation}{2}%
	\edef\theequation{\theparentequation\alph{frstequation},\alph{scndequation},\alph{trdequation}}
	\def\@currentlabel{\theparentequation\alph{frstequation}}\scndlabel{#1}
	\def\@currentlabel{\theparentequation\alph{scndequation}}\scndlabel{#2}
	\def\@currentlabel{\theparentequation\alph{trdequation}}\scndlabel{#3}
	\addtocounter{equation}{2}
}
\makeatother

\newcommand\BibTeX{{\rmfamily B\kern-.05em \textsc{i\kern-.025em b}\kern-.08em
		T\kern-.1667em\lower.7ex\hbox{E}\kern-.125emX}}

\articletype{Article Type}%

\begin{document}
	
	\title{Revisiting and Generalizing the Dual Iteration for Static and Robust Output-Feedback Synthesis}
	
	\author[1]{Tobias Holicki*}
	
	\author[1]{Carsten W. Scherer}
	
	
	\authormark{Holicki and Scherer}

	\address[1]{\orgdiv{Department of Mathematics}, \orgname{University of Stuttgart}, \orgaddress{\country{Germany}}}
	
	
	
	\corres{*Tobias Holicki, Department of Mathematics, University of Stuttgart, Pfaffenwaldring 5a, 70569 Stuttgart, Germany. \email{tobias.holicki@imng.uni-stuttgart.de}}
	
	
	\abstract[Abstract]{
		The dual iteration was introduced in a conference paper in 1997 by Iwasaki as an iterative and heuristic procedure for the challenging and non-convex design of static output-feedback controllers. We recall in detail its essential ingredients and go beyond the work of Iwasaki by demonstrating that the framework of linear fractional representations allows for a seamless extension of the dual iteration to output-feedback designs of practical relevance, such as the design of robust or robust gain-scheduled controllers. 
		
		In the paper of Iwasaki, the dual iteration is solely based on, and motivated by algebraic manipulations resulting from the elimination lemma. We provide a novel control theoretic interpretation of the individual steps, which paves the way for further generalizations of the powerful scheme to situations where the elimination lemma is not applicable. As an illustration, we extend the dual iteration to a design of static output-feedback controllers with multiple objectives. We demonstrate the approach with numerous numerical examples inspired from the literature.
	}
	
	
	\jnlcitation{\cname{%
	\author{T. Holicki}, and 
	\author{C. W. Scherer}} (\cyear{2021}), 
	\ctitle{Revisiting the Dual Iteration for Static and Robust Output-Feedback Synthesis}, \cjournal{Int J Robust Nonlinear Control}, \cvol{31;11:5427--5459}.}
	
	\maketitle
	
	
	
	\section{Introduction}\label{sec1}
	
	\begin{textblock}{16}(0.5, 14.7)
		\fbox{
			\begin{minipage}{0.9\textwidth}
				\fontsize{8pt}{9pt}\selectfont
				This is the peer reviewed version of the following article: [T. Holicki and C. W. Scherer, Revisiting and Generalizing the Dual Iteration for Static and Robust Output-Feedback Synthesis, \textit{Int. J. Robust Nonlin}, 2021; 31(11):5427-5459], which has been published in final form \href{https://doi.org/10.1002/rnc.5547}{here}. This article may be used for non-commercial purposes in accordance with Wiley Terms and Conditions for Use of Self-Archived Versions. This article may not be enhanced, enriched or otherwise transformed into a derivative work, without express permission from Wiley or by statutory rights under applicable legislation. Copyright notices must not be removed, obscured or modified. The article must be linked to Wiley’s version of record on Wiley Online Library and any embedding, framing or otherwise making available the article or pages thereof by third parties from platforms, services and websites other than Wiley Online Library must be prohibited.
			\end{minipage}
			}
	\end{textblock}

	The design of static output-feedback controllers constitutes a conceptually simple and yet theoretically very challenging problem. Such a design is also a popular approach of practical interest due to its straightforward implementation and the fact that, typically, only some (and not all) states of the underlying dynamical system are available for control. 
	However, in contrast to, e.g., the design of static state-feedback or dynamic full-order controllers, the synthesis of static output-feedback controllers is intrinsically a challenging bilinear matrix inequality (BMI) feasibility problem. Such problems are in general non-convex, non-smooth and NP-hard to solve \cite{TokOez95}. 
	%
	%
	%
	These troublesome properties have led to the development of a multitude of (heuristic) design approaches, which only yield sufficient conditions for the existence of such static controllers. 
	Next to providing only sufficient conditions, another downside of these approaches is that they might get stuck in a local minimum of the underlying optimization problem that can be far away from the global minimum of interest.
	Nevertheless, such approaches are employed and reported to work nicely on various practical examples. Two detailed surveys on static output-feedback design elaborating on several of such approaches are provided in \cite{SyrAbd97, SadPea16}.

	Similar difficulties arise in the general robust output-feedback controllers synthesis problem. Considering this general design problem is of tremendous relevance in practice since any designed controller is required to appropriately deal with 
	the mismatch between the employed model and the real system to be controlled. 
	The general strategy in \cite{ZhoDoy96, Sch01a, GreLim95} is to directly include uncertainty descriptions into the considered models. Depending on the real system, this often amounts to considering dynamical systems that are simultaneously affected by several uncertainties of different types, such as constant parametric, time-varying parametric, dynamic or nonlinear ones. 
	This calls for dedicated design methods with a very high flexibility, similarly as provided by the framework of integral quadratic constraints \cite{MegRan97} (IQCs) for system analysis. Unfortunately, most of the currently available methods lack this flexibility or are not very efficient. 

	\vspace{1ex}

	In this paper we present and extend the dual iteration which paves the way for resolving some of these issues. 
	The dual iteration is a heuristic method introduced in \cite{Iwa97, Iwa99} for designing stabilizing static output-feedback controllers for systems unaffected by uncertainties, which is also considered and conceptually compared to alternative approaches in the survey \cite{SadPea16}. 
	We elaborate in a tutorial fashion on the individual steps of this procedure for the design of static output-feedback $H_\infty$-controllers for linear time-invariant systems. In particular, we demonstrate that all underlying steps can be viewed as algebraic consequences of a general version of the elimination lemma as given, e.g., in \cite{Hel99}. 
	The latter lemma is a very powerful and flexible tool for controller design based on linear matrix inequalities (LMIs), which works perfectly well in tandem with the framework of linear fractional representations \cite{ZhoDoy96, SchWei00, Hof16} (LFRs). 
	As a consequence, this lemma enables us to provide a novel generalization of the dual iteration to a variety of challenging non-convex synthesis problems beyond the design of stabilizing static controllers as considered in \cite{Iwa97, Iwa99}. 
	In particular, we present a seamless extension of the dual iteration to robust $H_\infty$- and robust gain-scheduled $H_\infty$-design in the case that only output measurements are available for control; for these robust designs we consider arbitrarily time-varying parametric uncertainties and rely on IQCs with constant multipliers.
	
	Unfortunately, the elimination lemma does not apply for several interesting controller design problems such as those with multiple objectives, where it is also desirable to have an applicable variant of the dual iteration for static and/or robust design available. 
	To this end, we provide a control theoretic interpretation of the individual steps of the dual iteration, which does not involve the elimination lemma and builds upon \cite{HolSch19}. In \cite{HolSch19}, we have developed a heuristic approach for robust output-feedback design that was motivated by the well-known separation principle. This constitutes the consecutive solution of a full-information design problem and another design problem with a structure that resembles the one in robust estimation. We show that the latter is directly linked to the primal step of the dual iteration. 
	Based on this interpretation, we provide a generalization of the dual iteration to numerous situations where elimination is not possible. As a demonstration, we consider the design of a static output-feedback controller for an LTI system with two performance channels; the controller ensures that the first channel admits a small $H_\infty$-norm in closed-loop and that the second channel satisfies a quadratic performance criterion.
	
	\vspace{1ex}
	
	\noindent\textit{Outline.} %
	The remainder of the paper is organized as follows. After a short paragraph on notation, we recall in full detail the dual iteration for static output-feedback $H_\infty$-design in Sections \ref{SHI::sec::ana} and \ref{SHI::sec::synth}. 
	A novel control theoretic interpretation of the iteration's ingredients is then provided in Section \ref{SHI::sec::interpretation}. 
	We point out novel opportunities offered through this interpretation, by extending the dual iteration to the static output-feedback design of controllers with multiple objectives in Section \ref{SMO::sec::smo}.
	In Section \ref{RS::sec::rs} we show that the dual iteration is not limited to precisely known systems, by considering the practically highly relevant synthesis of robust output-feedback controllers for systems affected by arbitrarily time-varying uncertainties. Moreover, we also comment on further extensions of the iteration to deal, e.g., with the challenging synthesis of robust gain-scheduling controllers.
	The use of all these methods is demonstrated in terms of numerous numerical examples inspired from the literature, which includes a challenging missile autopilot design. Finally, several key auxiliary results are collected in the appendix.
	
	\vspace{1ex}
	\noindent\textit{Notation.} %
	%
	%
	%
	$L_2$ denotes the space of vector-valued square integrable functions with norm $\|x\|^2_{L_2} \!:=\! \int_0^\infty \!x(t)^T\! x(t)\,dt$.
	%
	%
	If $G(s) = D + C(sI - A)^{-1}B$, we write $G=[A, B, C, D]$,
	$G_{\ss} = \smat{A & B \\ C & D}$ and use $G^\ast = [-A^T,  C^T,  -B^T,  D^T]$ as well as $-G^\ast = [-A^T,  -C^T,  -B^T,  -D^T]$. 
	For matrices $A, B, C, D$, $X_1, \dots, X_N$, $X, P$, we employ the abbreviations $\mathrm{He}(X) := X + X^T$ and  
	\begin{equation*}
		\diag(X_1, \dots, X_N) := \mat{ccc}{X_1 & & 0 \\ & \ddots & \\ 0 & & X_N}
		\teq{ and }
		\Ls\left(X, P, \mat{cc}{A & B \\ C & D}\right) := \mat{cc}{I & 0 \\ A & B \\ \hline C & D}^T \mat{c|c}{X & 0 \\ \hline 0 & P}\mat{cc}{I & 0 \\ A & B \\ \hline C & D}.
	\end{equation*}
	%
	Finally, objects that can be inferred by symmetry or are not relevant are indicated by ``$\bullet$''.

	\section{Static Output-Feedback $H_\infty$-Design}\label{SHI::sec::shi}
	
	In this section, we recall the essential features of the dual iteration for static output-feedback design as proposed in \cite{Iwa97, Iwa99} in a tutorial fashion. In contrast to \cite{Iwa97, Iwa99} we directly include an $H_\infty$-performance criterion and, as the key point, we provide a novel control theoretic interpretation of the individual steps of the iteration. We reveal that this allows for interesting extensions as exemplified in the next section. 
	We begin by very briefly recalling the underlying definitions and analysis results.
	
	\subsection{Analysis}\label{SHI::sec::ana}
	
	For some real matrices of appropriate dimensions and some initial condition $x(0) \in \R^n$, we consider the system
	\begin{equation}
		\arraycolsep=3pt
		\mat{c}{\dot x(t) \\  e(t)} = \mat{cc}{A & B \\  C & D} \mat{c}{x(t)  \\ d(t)} \quad
		\label{SHI::eq::sys}
	\end{equation}
	for $t\geq 0$; here, $d\in L_2$ is a generalized disturbance and $e$ is the performance output which is desired to be small in the $L_2$-norm.
	The energy gain of the system \eqref{SHI::eq::sys} coincides with the $H_\infty$-norm of \eqref{SHI::eq::sys} and is defined in a standard fashion as follows.
	
	\begin{definition}
		\label{SHI::def::stab}
		The system \eqref{SHI::eq::sys} is said to admit an energy gain smaller than $\ga>0$ if $A$ is Hurwitz and there exists an $\eps > 0$ such that
		$\|e\|_{L_2}^2 \leq (\ga^2 - \eps) \|d\|_{L_2}^2$ for all $d \in L_2$ and for $x(0) = 0$.
		The energy gain of the system \eqref{SHI::eq::sys} is the infimal $\ga > 0$ such that the latter inequality is satisfied.
	\end{definition}
	\vspace{1ex}
	
	We have the following well-known analysis result which is often referred to as bounded real lemma (see, e.g., Section 2.7.3 of \cite{BoyGha94}) and constitutes a special case of the KYP lemma \cite{Ran96}.
	
	\begin{lemma}
		\label{SHI::lem::stab}
		Let $P_\ga := \smat{I & 0 \\ 0 & -\ga^2 I}$ and $G(s) := C(sI - A)^{-1}B + D$ be the transfer matrix corresponding to \eqref{SHI::eq::sys}. Then the system \eqref{SHI::eq::sys} admits an energy gain smaller than $\ga$ if and only if there exists a symmetric matrix $X$ satisfying
		\begin{equation}
			\label{SHI::lem::lmi_stab}
			X \cg 0
			\teq{ and }
			\Ls\left(\mat{cc}{0 & X \\ X & 0}, P_\ga, \mat{c}{G \\ I}_\ss \right) = 
			\mat{cc}{I & 0 \\ A & B}^T \mat{cc}{0 & X \\ X & 0}\mat{cc}{I & 0 \\ A & B} + \mat{cc}{C & D \\ 0 & I}^T P_\ga \mat{cc}{C & D \\ 0 & I} \cl 0.
			%
		\end{equation}
		Moreover, $\|G\|_{\infty} = \sup_{\omega \in \R }\|G(i\omega)\|$ equals the infimal $\ga > 0$ such that the above LMIs are feasible.
	\end{lemma}
	
	In our opinion, the abbreviation $\Ls(\cdot, \cdot, \cdot)$ in \eqref{SHI::lem::lmi_stab} is particularly well-suited for capturing the essential ingredients of inequalities related to the KYP lemma. Thus we make use of it throughout this paper. The involved symmetric matrix $X$ is usually referred to as a (KYP) certificate or as a Lyapunov matrix.

	
	\subsection{Synthesis}\label{SHI::sec::synth}
	
	\subsubsection{Problem Description}

	For fixed real matrices of appropriate dimensions and some initial conditions $x(0) \in \R^n$, we now consider the open-loop system%
	\begin{equation}
		\arraycolsep=1pt
		\mat{c}{\dot x(t) \\\hline  e(t) \\ y(t)} 
		= \mat{c|cc}{A & B_1 & B_2  \\ \hline C_1 & D_{11} & D_{12}  \\ C_2 & D_{21} & 0} 
		\mat{c}{x(t) \\\hline  d(t) \\ u(t)}
		\label{SHI::eq::sys_of}
	\end{equation}
	for $t\geq 0$; here, $u$ is the control input and $y$ is the measured output.
	Our main goal in this section is the design of a static output-feedback controller for the system  \eqref{SHI::eq::sys_of} with description
	\begin{equation}
		u(t) = K y(t)
		\label{SHI::eq::con_of2}
	\end{equation}
	such that the corresponding closed-loop energy gain is as small as possible. The latter closed-loop interconnection is given by%
	\begin{equation}
		\arraycolsep=2pt
		\mat{c}{\dot x(t) \\ e(t)} 
		= \mat{cc}{\Ac & \Bc  \\ 
			\Cc & \Dc } 
		\mat{c}{x(t)\\ d(t)}
		= \left( \mat{cc}{A & B_1 \\ C_1 & D_{11}} + \mat{c}{B_2 \\ D_{12}}K \mat{cc}{C_2 & D_{21}}\right) \mat{c}{x(t)\\ d(t)}
		\label{SHI::eq::cl_of}
	\end{equation}
	with $t\geq 0$.
	Note that the system \eqref{SHI::eq::cl_of} is of the same form as \eqref{SHI::eq::sys}, which allows for its analysis based on the bounded real lemma \ref{SHI::lem::stab}. 
	As usual, trouble arises through the simultaneous search for some certificate $X$ and a controller gain $K$, which is a very difficult non-convex BMI problem. 
	A remedy for a multitude of controller synthesis problems is a convexifying parameter transformation that has been proposed in \cite{MasOha98, Sch96b}. Another option is given by the elimination lemma as developed in \cite{Hel99, GahApk94}. 
	The latter lemma is well-known in the LMI literature, but since we will apply it frequently, we provide the result as Lemma \ref{RS::lem::elimination} together with a constructive proof in the appendix.
	In particular, by directly using the elimination lemma on the closed-loop analysis LMIs, we immediately obtain the following well-known result.
	

	\begin{theorem}
		\label{SHI::theo::of}
		Let $G_{11} := [A, B_1, C_1, D_{11}]$. Further, let $V$ and $U$ be basis matrices of $\ker(C_2, D_{21})$ and $\ker(B_2^T, D_{12}^T)$, respectively. Then there exists a static controller \eqref{SHI::eq::con_of2} for the system \eqref{SHI::eq::sys_of} such that the analysis LMIs \eqref{SHI::lem::lmi_stab} are feasible for the corresponding closed-loop system if and only if there exists a symmetric matrix $X$ satisfying 
		\begin{equation*}
			\arraycolsep=3pt
			X \cg 0, \quad
			V^T\Ls\left(\mat{cc}{0 & X \\ X & 0}, P_\ga, \mat{c}{G_{11} \\ I
			}_{\ss} \right) V \cl 0
			\teq{ and }
			U^T\Ls\left(\mat{cc}{0 & X^{-1} \\ X^{-1} & 0}, P_\ga^{-1}, \mat{c}{I  \\ -G_{11}^\ast}_{\ss} \right) U \cg 0.
		\end{equation*}
		Moreover, the infimal $\ga > 0$ such that there exists some symmetric $X$ satisfying the above inequalities is equal to
		\begin{equation*}
			\ga_\opt := \inf\left\{ \ga > 0~\middle| 
			\text{There exists a controller \eqref{SHI::eq::con_of2} s.th. the analysis LMIs \eqref{SHI::lem::lmi_stab} are feasible in $X$ for the interconnection \eqref{SHI::eq::cl_of}
			}\right\}.
		\end{equation*}
	\end{theorem}
	
	By the elimination lemma we are able to remove the controller gain $K$ from the analysis LMIs for the closed-loop system \eqref{SHI::eq::cl_of}. However, the variable $X$ now enters the above inequalities in a non-convex fashion. Therefore, determining $\ga_\opt$ or computing a suitable static controller \eqref{SHI::eq::con_of2} remain difficult.
	%
	Note that this underlying non-convexity is not limited to the employed elimination based approach, but seems to be an intrinsic feature of the static controller synthesis problem. 
	Thus the latter problem is usually tackled by heuristic approaches, and upper bounds on $\ga_\opt$ are computed.
	%
	In the sequel, we present the dual iteration from \cite{Iwa97, Iwa99} which is a heuristic procedure based on iteratively solving convex semi-definite programs. We will argue that this iteration is especially useful if compared to other approaches since it provides good upper bounds on $\ga_\opt$ and since it seamlessly generalizes, for example, to robust design problems.
	Its essential features are discussed next.

	\subsubsection{Dual Iteration: Initialization}\label{SHI::sec::dual_init}
	
	In order to initialize the dual iteration, we propose a starting point that allows the computation of a lower bound on $\ga_\opt$ as a valuable indicator of how conservative any later computed upper bound on $\ga_\opt$ is.
	%
	This lower bound is obtained by the following observation. If there exists a static controller \eqref{SHI::eq::con_of2} for the system \eqref{SHI::eq::sys_of} achieving a closed-loop energy gain of $\ga$, then there also exists a dynamic\footnote{In this paper and for brevity, dynamic controllers is are always of full-order, i.e., they have the same number of states as the underlying open-loop system.} controller with description
	\begin{equation}
		\mat{c}{\dot x_c(t)  \\ u(t)} 
		= \mat{cc}{A^c & B^c \\
			C^c & D^c} \mat{c}{x_c(t)  \\ y(t)}
		\label{SHI::eq::con_of3}
	\end{equation}
	which achieves (at least) the same closed-loop energy gain. Indeed, by simply choosing $A^c = -I_n$, $B^c = 0$, $C^c = 0$ and $D^c = K$, we observe that the energy gain of \eqref{SHI::eq::cl_of} is identical to the one of the closed-loop interconnection of the system \eqref{SHI::eq::sys_of} and the dynamic controller \eqref{SHI::eq::con_of3}. Note that, the matrix $-I_n$ can be replaced by any other stable matrix in $\R^{n \times n}$.
	It is well-known that the problem of finding a dynamic controller \eqref{SHI::eq::con_of3} for the system \eqref{SHI::eq::sys_of} is a convex optimization problem with the following solution which is, again, obtained by applying the elimination lemma \ref{RS::lem::elimination}. A proof can also be found, e.g., in \cite{GahApk94, IwaSke94}.

	\begin{theorem}
		\label{SHI::theo::gs}
		Let $G_{11}$, $U$ and $V$ be as in Theorem \ref{SHI::theo::of}. Then there exists a dynamic controller \eqref{SHI::eq::con_of3} for the system \eqref{SHI::eq::sys_of} such that the analysis LMIs \eqref{SHI::lem::lmi_stab} are feasible for the corresponding closed-loop system if and only if there exist symmetric matrices $X$ and $Y$ satisfying 
		\begin{subequations}
			\label{SHI::theo::eq::lmi_gs}
			\begin{equation}
				\arraycolsep=3pt
				\mat{cc}{X & I \\ I & Y} \cg 0,\quad
				V^T\Ls\left(\mat{cc}{0 & X \\ X & 0}, P_\ga, \mat{c}{G_{11} \\ I 
				}_{\ss} \right) V \cl 0
				\teq{ and }
				U^T\Ls\left(\mat{cc}{0 & Y \\ Y & 0}, P_\ga^{-1}, \mat{c}{I  \\ -G_{11}^\ast 
				}_{\ss} \right) U \cg 0.
				\tlabel{SHI::theo::eq::lmi_gsa}{SHI::theo::eq::lmi_gsb}{SHI::theo::eq::lmi_gsc}
			\end{equation}
		\end{subequations}
		In particular, we have
		$\ga_{\mathrm{dof}} \leq \ga_\opt$
		for $\ga_\mathrm{dof}$ being the infimal $\ga > 0$ such that 
		the  LMIs \eqref{SHI::theo::eq::lmi_gs} are feasible.
	\end{theorem}
	
	In a standard fashion and by using the Schur complement on the LMI \eqref{SHI::theo::eq::lmi_gsc}, it is possible to solve the LMIs \eqref{SHI::theo::eq::lmi_gs} while simultaneously minimizing over $\ga$ in order to compute $\ga_\mathrm{dof}$. In particular, as the latter is a lower bound on $\ga_\opt$ it is not possible to find a static output-feedback controller with an energy gain smaller than $\ga_\mathrm{dof}$. 
	
	\vspace{2ex}
	
	As an intermediate step, if the LMIs \eqref{SHI::theo::eq::lmi_gs} are feasible, we note that we can easily design a static full-information controller $u = F\t y = (F_1, F_2)\t y$ for the system \eqref{SHI::eq::sys_of} such that the analysis LMIs \eqref{SHI::lem::lmi_stab}  are feasible for the corresponding closed-loop system; here, the measurements $y$ are replaced by the virtual measurements $\t y := \smat{x \\ d}$ and the resulting interconnection is explicitly given by
	\begin{equation}
		\mat{c}{\dot x(t) \\ e(t)}
		= \mat{cc}{A + B_2 F_1 & B_1 + B_2F_2 \\ 
			C_1 + D_{12}F_1 & D_{11} + D_{12}F_2}
		\mat{c}{x(t)  \\ d(t)}
		= \left(\mat{cc}{A& B_1 \\ C_1 & D_{11}}
		+ \mat{c}{B_2 \\ D_{12}} F
		\right)
		\mat{c}{x(t)  \\ d(t)}.
		\label{SHI::eq::clF}
	\end{equation}
	Indeed, by applying the elimination lemma \ref{RS::lem::elimination}, we immediately obtain the following convex synthesis result.
	
	\begin{lemma}
		\label{SHI::lem::full_info}
		There exists some full-information gain $F$ such that the analysis LMIs \eqref{SHI::lem::lmi_stab} are feasible for the system \eqref{SHI::eq::clF} if and only if there exists a symmetric matrix $Y \cg 0$ satisfying \eqref{SHI::theo::eq::lmi_gsc}.
	\end{lemma}

	\subsubsection{Dual Iteration}\label{SHI::sec::dual}

	We are now in the position to discuss the core of the dual iteration from \cite{Iwa97, Iwa99}. 
	The first key result provides LMI conditions that are sufficient for static output-feedback design based on the assumption that a full-information gain $F = (F_1, F_2)$ is available. 

	\begin{theorem}
		\label{SHI::theo::ofF}
		Let $G_{11}$ and $V$ be as in Theorem \ref{SHI::theo::of} and let $G^F_{11}$ be the transfer matrix corresponding to \eqref{SHI::eq::clF}. Further, suppose that $A + B_2F_1$ is stable. Then there exists a static controller \eqref{SHI::eq::con_of2} for the system \eqref{SHI::eq::sys_of} such that the analysis LMIs \eqref{SHI::lem::lmi_stab} are feasible for the corresponding closed-loop system if there exists a symmetric matrix $X$ satisfying
		\begin{subequations}
			\label{SHI::theo::eq::lmi_ofF}
			\begin{equation}
				\arraycolsep=3pt
				%
				V^T\Ls\left(\mat{cc}{0 & X \\ X & 0},  P_\ga, \mat{c}{G_{11}  \\ I 
				}_{\ss} \right) V \cl 0
				\teq{ and }
				\Ls\left(\mat{cc}{0 & X \\ X & 0},  P_\ga, \mat{c}{G^F_{11} \\ I 
				}_{\ss} \right) \cl 0.
				\dlabel{SHI::theo::eq::lmi_ofFa}{SHI::theo::eq::lmi_ofFb}
			\end{equation}
		\end{subequations}
		Moreover, we have
		$\ga_\mathrm{dof} \leq \ga_\opt \leq \ga_F$
		for $\ga_F$ being the infimal $\ga > 0$ such that the LMIs \eqref{SHI::theo::eq::lmi_ofF} are feasible.
	\end{theorem}

	\begin{proof}
		The left upper block of \eqref{SHI::theo::eq::lmi_ofFb} reads as the Lyapunov inequality
		\begin{equation*}
			(A + B_2F_1)^TX + X(A + B_2F_1) + (C_1 + D_{12}F_1)^T(C_1 + D_{12}F_1) \cl 0.
		\end{equation*}
		Hence stability of $A + B_2F_1$ implies $X \cg 0$. This enables us to apply the elimination lemma \ref{RS::lem::elimination} in order to remove the full-information controller gain $F$ from the LMI \eqref{SHI::theo::eq::lmi_ofFb}, which yields exactly the third of the inequalities in Theorem \ref{SHI::theo::of}. Combined with $X \cg 0$ and \eqref{SHI::theo::eq::lmi_ofFa}, this allows us to construct the desired static controller via Theorem \ref{SHI::theo::of}. 
	\end{proof}

	Observe that $A + B_2 F_1$ is stable by construction if the gain $F$ is designed based on Lemma \ref{SHI::lem::full_info}, and that \eqref{SHI::theo::eq::lmi_ofFb} exactly is the analysis LMI \eqref{SHI::lem::lmi_stab} for the interconnection \eqref{SHI::eq::clF}.  Further, note that we even have $\ga_\opt = \ga_F$ if we view the gain $F$ as a decision variable in \eqref{SHI::theo::eq::lmi_ofF}. However, this would render the computation of $\ga_F$ as troublesome as that of $\ga_\opt$ itself.

	Intuitively, Theorem \ref{SHI::theo::ofF} links the difficult static output-feedback and the manageable full-information design problem with a common quantity (here, the Lyapunov matrix $X$). 
	This underlying idea is also employed in order to deal with many other non-convex and/or difficult problems such as the ones considered in \cite{EbiPea15, ArzPea03, HenSeb03}. 
	In fact, one can even show that there exist some gain $F$ and a matrix $X \cg 0$ satisfying \eqref{SHI::theo::eq::lmi_ofF} if and only if there exist matrices $X, F, N = (N_1, N_2)$ satisfying
	\begin{equation}
		\label{SHI::eq::svar_lmi}
		\arraycolsep=4pt
		X \cg 0
		\teq{ and }
		\Ls\left(\mat{cc}{0 & X \\ X & 0}, P_\ga, \mat{ccc}{A & B_1 & B_2 \\ \hline C_1 & D_{11} & D_{12} \\ 0 & I & 0} \right) + \mathrm{He}\left(\mat{c}{F^T \\ -I}\mat{ccc}{N_1 C_2 & N_1 D_{21} & -N_2} \right) \cl 0.
	\end{equation}
	The latter inequalities form the basis of the approach in \cite{EbiPea15} for static output-feedback $H_\infty$-design.
	
	\vspace{1ex}
	
	While Theorem \ref{SHI::theo::ofF} is interesting on its own, the key idea of the dual iteration is that improved upper bounds on $\ga_\opt$ are obtained by 
	also considering a problem that is dual to full-information synthesis.
	This consists of finding a full-actuation gain $E = (E_1^T, E_2^T)^T$ such that the analysis LMIs \eqref{SHI::lem::lmi_stab} are feasible for the system
	\begin{equation}
		\mat{c}{\dot x(t)  \\ e(t)}
		= \mat{cc}{A + E_1 C_2 & B_1 + E_1D_{21} \\
			C_1 + E_2C_2 & D_{11} + E_2D_{21} }
		\mat{c}{ x(t)  \\ d(t)}
		= \left(\mat{cc}{A & B_1  \\
			C_1 & D_{11} }
		+ E \mat{cc}{C_2 & D_{21}} 
		\right)
		\mat{c}{ x(t)  \\ d(t)}.
		\label{SHI::eq::clE}
	\end{equation}
	As before, a convex solution in terms of LMIs is immediately obtained by the elimination lemma \ref{RS::lem::elimination} and reads as follows.

	\begin{lemma}
		\label{SHI::lem::full_actu}
		There exists some full-actuation gain $E$ such that the analysis LMIs \eqref{SHI::lem::lmi_stab} are feasible for the system \eqref{SHI::eq::clE} if and only if there exists a symmetric matrix $X \cg 0$ satisfying \eqref{SHI::theo::eq::lmi_gsb}.
	\end{lemma}
	
	Based on a designed full-actuation gain $E$ we can formulate another set of LMI conditions that are sufficient for static output-feedback design. The proof is analogous to the one of Theorem \ref{SHI::theo::ofF} and can hence be omitted.
	
	\begin{theorem}
		\label{SHI::theo::ofE}
		Let $G_{11}$ and $U$ be as in Theorem \ref{SHI::theo::of} and let $G^E_{11}$ be the transfer matrix corresponding to \eqref{SHI::eq::clE}. Further, suppose that $A + E_1C_2$ is stable. Then there exists a static controller \eqref{SHI::eq::con_of2} for the system \eqref{SHI::eq::sys_of} such that the analysis LMIs \eqref{SHI::lem::lmi_stab} are feasible for the corresponding closed-loop system if there exists a symmetric matrix $Y$ satisfying 
		\begin{subequations}
			\label{SHI::theo::eq::lmi_ofE}
			\begin{equation}
				\arraycolsep=3pt
				\Ls\left(\mat{cc}{0 & Y \\ Y & 0},  P_\ga^{-1}, \mat{c}{I  \\ -(G^E_{11})^\ast 
				}_{\ss} \right) \cg 0
				\teq{ and }
				U^T\Ls\left(\mat{cc}{0 & Y \\ Y & 0},  P_\ga^{-1}, \mat{c}{I  \\ -G_{11}^\ast 
				}_{\ss} \right) U \cg 0.
				\dlabel{SHI::theo::eq::lmi_ofEa}{SHI::theo::eq::lmi_ofEb}
			\end{equation}
		\end{subequations}
		Moreover, we have
		$\ga_\mathrm{dof} \leq \ga_\opt \leq \ga_E$
		for $\ga_E$ being the infimal $\ga > 0$ such that the LMIs \eqref{SHI::theo::eq::lmi_ofE} are feasible.
	\end{theorem}
	
	In the sequel we refer to the LMIs \eqref{SHI::theo::eq::lmi_ofF} and \eqref{SHI::theo::eq::lmi_ofE} as primal and dual synthesis LMIs, respectively. 
	Accordingly, we address Theorems \ref{SHI::theo::ofF} and \ref{SHI::theo::ofE} as primal and dual design results, respectively.
	Observe that the latter are nicely intertwined as follows.
	
	\begin{theorem}
		\label{SHI::theo::it_summary}
		The following two statements hold.
		\begin{itemize}
			\item If $A + B_2F_1$ is stable and the primal synthesis LMIs \eqref{SHI::theo::eq::lmi_ofF} are satisfied for some $\ga$, some matrix $X$ and some full-information gain $F=(F_1, F_2)$, then there exists some full-actuation gain $E = (E_1^T, E_2^T)^T$ such that $A+E_1C_2$ is stable and the dual synthesis LMIs \eqref{SHI::theo::eq::lmi_ofE} are satisfied for the same $\ga$ and for $Y = X^{-1}$. In particular, we have $\ga_E < \ga$.
			\item If $A + E_1C_2$ is stable and the dual synthesis LMIs \eqref{SHI::theo::eq::lmi_ofE} are satisfied for some $\ga$, some matrix $Y$ and some full-actuation gain $E = (E_1^T, E_2^T)^T$, then there exists some full-information gain $F = (F_1, F_2)$ such that $A + B_2F_1$ is stable and the primal synthesis LMIs \eqref{SHI::theo::eq::lmi_ofF} are satisfied for the same $\ga$ and for $X = Y^{-1}$. In particular, we have $\ga_F < \ga$.
		\end{itemize}
	\end{theorem}
	
	\begin{proof}
		We only show the first statement as the second one follows with analogous arguments. If $A + B_2F_1$ is stable and the primal synthesis LMIs \eqref{SHI::theo::eq::lmi_ofF} are feasible, we can infer $X \cg 0$ from \eqref{SHI::theo::eq::lmi_ofFb} as in Theorem \ref{SHI::theo::ofF}. Due to \eqref{SHI::theo::eq::lmi_ofFa} and Lemma \ref{SHI::lem::full_actu}, we can then conclude the existence of a full-actuation gain $E$ satisfying
		\begin{equation*}
			\Ls\left(\mat{cc}{0 & X \\ X & 0},  P_\ga, \mat{c}{G^E_{11}  \\ I}_\ss \right) \cl 0
		\end{equation*}
		with exactly the same Lyapunov matrix $X\cg 0$. In particular, the left upper block of the above LMI is a standard Lyapunov inequality which implies that $A + E_1C_2$ is stable. Moreover, an application of the dualization lemma \ref{RS::lem::dualization} as given in the appendix allows us to infer that \eqref{SHI::theo::eq::lmi_ofEa} is satisfied for $Y = X^{-1} \cg 0$. Finally, by using the elimination lemma \ref{RS::lem::elimination} on the LMI \eqref{SHI::theo::eq::lmi_ofFb} to remove the full-information gain $F$, we conclude that \eqref{SHI::theo::eq::lmi_ofEb} is satisfied as well. This finishes the proof.
	\end{proof}

	The dual iteration now essentially amounts to alternately applying the two statements in Theorem~\ref{SHI::theo::it_summary} and is stated as follows.

	\begin{Algorithm}
		\label{SHI::algo::dual_iteration}
		Dual iteration for static output-feedback $H_\infty$-design.
		\begin{enumerate}
			\item \emph{Initialization:} Compute the lower bound $\ga_\mathrm{dof}$ based on solving the dynamic synthesis LMIs \eqref{SHI::theo::eq::lmi_gs} and set $\ga^0 := +\infty$ as well as $k = 1$.
			Design an initial full-information gain $F$ from Lemma \ref{SHI::lem::full_info}.
			\item \emph{Primal step:} Compute $\ga_F$ by solving the primal synthesis LMIs \eqref{SHI::theo::eq::lmi_ofF} for the given gain $F$ and choose some small $\eps_k>0$ such that $\ga^k := \ga_F(1+\eps_k) < \ga^{k-1}$. For $\ga = \ga^k$, determine a matrix $X$ satisfying the LMIs \eqref{SHI::theo::eq::lmi_ofF} and apply the elimination lemma \ref{RS::lem::elimination} on \eqref{SHI::theo::eq::lmi_ofFa} in order to construct a full-actuation gain $E$ satisfying the dual synthesis LMIs \eqref{SHI::theo::eq::lmi_ofE} for $Y = X^{-1}$. 
			\item \emph{Dual step:} Compute $\ga_E$ by solving the dual synthesis LMIs \eqref{SHI::theo::eq::lmi_ofE} for the given gain $E$ and choose some small $\eps_{k+1}> 0$ such that $\ga^{k+1} := \ga_E (1+\eps_{k+1}) < \ga^k$. For $\ga = \ga^{k+1}$, determine a matrix $Y$ satisfying the LMIs \eqref{SHI::theo::eq::lmi_ofE} and apply the elimination lemma \ref{RS::lem::elimination} on \eqref{SHI::theo::eq::lmi_ofEa} in order to construct a full-information gain $F$ satisfying the primal synthesis LMIs  \eqref{SHI::theo::eq::lmi_ofF} for $X = Y^{-1}$.
			\item \emph{Termination:} If $k$ is too large or $\ga^k$ does not decrease any more, then stop and construct a static output-feedback controller according to Theorem \ref{SHI::theo::ofE}.\\
			Otherwise set $k = k+2$ and go to the primal step. 
		\end{enumerate}
	\end{Algorithm}

	\begin{remark}
		~
		\label{SHI::rema::stuff}
		\begin{enumerate}[(a)]
			\item Theorem \ref{SHI::theo::it_summary} ensures that Algorithm \ref{SHI::algo::dual_iteration} is recursively feasible, i.e., it will not get stuck due to infeasibility of some LMI, if the primal synthesis LMIs \eqref{SHI::theo::eq::lmi_ofF} are feasible when performing the primal step for the first time. Additionally, the proof of Theorem \ref{SHI::theo::it_summary} demonstrates that we can even warm start the feasibility problems in the primal and dual steps by providing a feasible initial guess for the involved variables. This reduces the computational burden remarkably. 
			%
			%
			
			\item The small numbers $\eps_k> 0$ are introduced since, in general, it is not possible to determine optimal controllers or gains (because these might not even exist); this is the reason for working with close-to-optimal solutions instead.
			
			\item We have
			$\ga_\mathrm{dof} \leq \ga_\opt \leq \ga^k < \dots < \ga^2 < \ga^1$
			for all $k\in \N$
			and thus the sequence $(\ga^k)_{k\in \N}$ converges to some value $\ga^\ast \geq \ga_\opt$. 
			As for other approaches, there is no guarantee that $\ga^\ast = \ga_\opt$. 
			Nevertheless, the number of required iterations to obtain acceptable bounds on the optimal energy gain is rather low as will be demonstrated. 
			\item As for any heuristic design, it can be beneficial to perform an a posteriori closed-loop analysis via Lemma \ref{SHI::lem::stab}. The resulting closed-loop energy gain is guaranteed to be not larger than the corresponding computed upper bound $\ga^k$.
			%
			%
			\item If a static controller $K$ is available that achieves a closed-loop energy gain bounded by $\ga$, then the dual iteration can be initialized with $F = (KC_2, KD_{21})$. In particular, the primal synthesis LMIs \eqref{SHI::theo::eq::lmi_ofF} are then feasible and we have $\ga_F \leq \ga$.
			
			
			%
			%
			\item If one is only interested in stability as in the original publication \cite{Iwa97, Iwa99}, one should replace the analysis LMIs \eqref{SHI::lem::lmi_stab} with $X \cg 0$ and $A^TX + XA \cl \ga X$
			and adapt the design results accordingly while still minimizing $\ga$. 
			In this case, we note that the emergence of terms like $\ga X$ requires the solution of generalized eigenvalue problems. These can be efficiently solved as well, e.g., with Matlab and LMIlab \cite{GahNem95}.
		\end{enumerate}
	\end{remark}

	\begin{remark}
		\label{SHI::rema::init}
		The selection of a suitable gain $F$ during the initialization of Algorithm \ref{SHI::algo::dual_iteration} can be crucial, since feasibility of the primal synthesis LMIs \eqref{SHI::theo::eq::lmi_ofF} is \emph{not} guaranteed from the  feasibility of dynamic synthesis LMIs \eqref{SHI::theo::eq::lmi_gs} and depends on the concrete choice of the gain $F$. 
		Similarly as in \cite{Iwa99}, we propose to compute the lower bound $\ga_\mathrm{dof}$ and then to reconsider the LMIs \eqref{SHI::theo::eq::lmi_gs} for $\ga = (1+\eps) \ga_\mathrm{dof}$ and some fixed $\eps > 0$ while minimizing $\tr(X + Y)$.  Due to \eqref{SHI::theo::eq::lmi_gsa}, this is a common heuristic that aims to push $X$ towards $Y^{-1}$ and which promotes feasibility of the non-convex design matrix inequalities in Theorem \ref{SHI::theo::of}. Constructing a gain $F$ based on Lemma \ref{SHI::lem::full_info} and these modified LMIs promotes feasibility of the primal synthesis LMIs \eqref{SHI::theo::eq::lmi_ofF} as well.
	\end{remark}

	\subsection{Examples}\label{SHI::sec::exa}

	In order to illustrate the dual iteration from \cite{Iwa97, Iwa99} and as described above, we consider several examples from COMPl\textsubscript{e}ib \cite{Lei04} and compare the iteration to the following common and recent alternative static output-feedback $H_\infty$-design approaches:
	\begin{itemize}
		\item A D-K iteration scheme (also termed V-K iteration, e.g., in \cite{Boy94, GhaBal94}) that relies on minimizing $\ga$ subject to the analysis LMIs \eqref{SHI::lem::lmi_stab} for the closed-loop system \eqref{SHI::eq::cl_of} (with decision variables $\ga, X, K$) while alternately fixing $K$ and $X$. We emphasize that this approach \emph{requires} an initialization with a stabilizing static controller, because the first considered LMI is infeasible otherwise. In this paper, we utilize the static controller as obtained from computing $\ga^1$ for the initialization.
		\item The approach presented in Section 6.3 of \cite{EbiPea15}, which makes use of so-called ``S-Variables'', will be referred to as SVar iteration. This approach is based on minimizing $\ga$ subject to the LMIs \eqref{SHI::eq::svar_lmi} (with decision variables $\ga, X, N, F$) while alternately fixing $F$ and $N$. We initialize this algorithm in the same way as the dual iteration and as stated in Remark \ref{SHI::rema::init}.
		\item The \texttt{hinfstruct} algorithm from \cite{ApkNol06} available in Matlab using default options.
		\item \texttt{hifoo 3.5} with \texttt{hanso 2.01} from \cite{BurHen06} using default options.
	\end{itemize}
	We denote the resulting upper bounds on $\ga_\opt$ as $\ga_\mathrm{dk}^k$, $\ga_\mathrm{svar}^k$, $\ga_{\mathrm{his}}$ and $\ga_{\mathrm{hfo}}$, respectively; the superscript $k$ indicates that the algorithm was stopped after $k$ iterations. All computations are carried out with Matlab on a general purpose desktop computer (Intel Core i7, 4.0 GHz, 8 GB of ram) and we use LMIlab \cite{GahNem95} for solving LMIs; 
	in our experience the latter solver is not the fastest, but the most reliable one that is available for LMI based controller design. The Matlab code for all example in this paper is available in \cite{HolSch21}.

	The numerical results are depicted in Table~\ref{SHI::tab::results_sof} and do not show dramatic differences between the dual iteration, \texttt{hinfstruct} and \texttt{hifoo} in terms of computed upper bounds for most of the examples. 
	However, the dual iteration clearly outperforms the D-K and the SVar iteration. Similarly as in the original publication \cite{Iwa99}, we observe that few iterations of the dual iteration are often sufficient to obtain good upper bounds on the optimal $\ga_\opt$, which is in contrast to the latter two algorithms. Obviously, all of the iterations can lead to potential improvements for more than the chosen nine iterations.
	Finally, note that all of the considered algorithms can fail to provide a stabilizing solution, which is due to the underlying non-convexity of the synthesis problem. 
	In this case or in order to potentially improve the obtained upper bounds, \texttt{hinfstruct} offers the possibility to restart with randomly chosen initial conditions. This is as well possible for the dual and the SVar iteration, 
	for example with a strategy as described in Section 6.5 of \cite{EbiPea15}. The algorithm behind \texttt{hifoo} is randomized by itself and can thus also profit from performing multiple runs. Clearly, all of these restart strategies come at the expense of additional computational time.

	The numbers $T_{\ga^9}$ $T_{\ga_{\mathrm{dk}}^9}$, $T_{\ga_{\mathrm{svar}}^9}$, $T_{\ga_{\mathrm{his}}}$ and $T_{\ga_{\mathrm{hfo}}}$ in Table~\ref{SHI::tab::results_sof} denote the average runtime for twenty runs in seconds required for the computation of $\ga^9$, $\ga_{\mathrm{dk}}^9$, $\ga_{\mathrm{svar}}^9$, $\ga_{\mathrm{his}}$ and $\ga_{\mathrm{hfo}}$, respectively.
	We observe that our implementation of the dual iteration is mostly slower than the D-K and the SVar iteration.  Moreover, it is faster than \texttt{hinfstruct} and \texttt{hifoo} for systems with a small number of states $n$, but does not scale well for systems with many states. 
	The latter is, of course, not surprising since the dual iteration is based on solving LMIs and thus inherits all related computational aspects. In contrast, \texttt{hinfstruct} and \texttt{hifoo} rely on a more specialized optimization techniques that avoid solving LMIs. 
	Note that the required computation time for the dual iteration can easily be improved by applying faster generic LMI solvers such as Mosek \cite{Mos17} or SeDuMi \cite{Stu01}. Instead of relying on generic solvers, it might even be possible to employ dedicated solvers such as \cite{WalHan04} which exploit the particular structure of the primal and dual synthesis LMIs \eqref{SHI::theo::eq::lmi_ofF} and \eqref{SHI::theo::eq::lmi_ofE}. However, exploring this potential for numerical improvements is beyond the scope of this paper. Finally, note that the initialization of the dual iteration is the most time-consuming part; the actual iteration is relatively fast in comparison, since fewer decision variables are involved.

	\vspace{1ex}
	
	The most important benefit of LMI based approaches (such as the dual iteration) over algorithms such as \texttt{hinfstruct} or \texttt{hifoo} is their potential for generalizations to deal with problems that are more interesting and relevant than 
	the mere design of static output-feedback controllers as considered in this section. 
	These problems include the design of (robust) controllers for systems involving delayed signals and uncertain or nonlinear components, the synthesis of static controllers for time-varying systems as well as the design of controllers for systems with continuous and discrete dynamics.
	In particular, we demonstrate in Section \ref{RS::sec::rs} that the dual iteration allows us to synthesize, within a common framework, robust and robust gain-scheduling controllers for systems affected by time-varying parametric uncertainties (and scheduling components).
	Finally, note again that a more detailed discussion and conceptual comparison of static design approaches involving many more algorithms is found in the recent survey \cite{SadPea16}.

	\begin{table}
		\newcommand{\f}[1]{\bf{#1}}
		\newcommand{\gr}[1]{\textcolor{gray}{#1}}
		\caption{Optimal closed-loop $H_\infty$-norms resulting from dynamic output-feedback design with upper bounds obtained via the dual iteration, a D-K iteration, the SVar iteration, \texttt{hinfstruct} and \texttt{hifoo} for several examples from \cite{Lei04}, together with the corresponding average running times for twenty runs in seconds. All values are rounded to two decimals.}%
		\label{SHI::tab::results_sof}%
		
		\begin{center}%
			\scalebox{0.925}{%
				\setlength{\tabcolsep}{5pt}%
				\renewcommand{\arraystretch}{1.2}%
				\begin{tabular}{@{}l@{\hskip 2ex}r@{\hskip 3.5ex}rrrr@{\hskip 3.5ex}rrr@{\hskip 3.5ex}rrr@{\hskip 3.5ex}rr@{\hskip 3.5ex}rr@{}}
					\toprule
					& &\multicolumn{4}{@{}c@{\hskip 3.5ex}}{Dual Iteration} & \multicolumn{3}{@{}c@{\hskip 3.5ex}}{D-K Iteration} & \multicolumn{3}{@{}c@{\hskip 3.5ex}}{SVar Iteration} & \multicolumn{2}{@{}c@{\hskip 3.5ex}}{\texttt{hinfstruct}} & \multicolumn{2}{@{}c}{\texttt{hifoo} }\\ \cmidrule(r{3.5ex}){3-6}\cmidrule(r{3.5ex}){7-9}\cmidrule(r{3.5ex}){10-12} \cmidrule(r{3.5ex}){13-14} \cmidrule(r){15-16}
					Name & $\ga_\mathrm{dof}$ & $\ga^1$ & $\ga^5$ & $\ga^9$ & $T_{\ga^9}$ & $\ga_\mathrm{dk}^5$ & $\ga_\mathrm{dk}^{9}$ & $T_{\ga_{\mathrm{dk}}^9}$ & $\ga_{\mathrm{svar}}^5$ & $\ga_{\mathrm{svar}}^9$ & $T_{\ga_{\mathrm{svar}}^9}$ & $\ga_\mathrm{his}$ & $T_{\ga_{\mathrm{his}}}$ &$\ga_{\mathrm{hfo}}$ & $T_{\ga_{\mathrm{hfo}}}$\\ \hline 
					AC3 & 2.97 & 4.53 & 3.67 & \f{3.47} & \gr{0.10} & 4.12 & 4.03 & \gr{0.08}  & 3.90 & 3.83 & \gr{0.09} & 3.64 & \gr{0.16} & 3.62 & \gr{3.03}\\
					AC18 & 5.38 & 14.62 & 10.74 & 10.72 & \gr{1.37} & 12.22 & 11.95 & \gr{1.28}  & 11.20 & 11.12 & \gr{1.32} & \f{10.70} & \gr{0.15} & 12.66 & \gr{4.03}\\
					HE2 & 2.42 & 5.28 & 4.26 & 4.25 & \gr{0.08} & 4.97 & 4.94 & \gr{0.06}  & 4.40 & 4.27 & \gr{0.06} & 4.25 & \gr{0.11} & \f{4.14} & \gr{0.89}\\
					HE4 & 22.84 & 32.34 & 23.02 & \f{22.84} & \gr{0.56} & 31.25 & 30.56 & \gr{0.53}  & 26.41 & 24.81 & \gr{0.6} & 23.57 & \gr{0.36} & 22.85 & \gr{33.13}\\
					JE1 & 3.85 & 20.40 & 12.42 & 11.70 & \gr{1.28k} & 19.15 & 18.62 & \gr{1.16k}  & 16.95 & 15.08 & \gr{1.12k} & \f{10.15} & \gr{1.72} & 23.52 & \gr{49.48}\\
					REA2 & 1.13 & 1.24 & 1.17 & 1.16 & \gr{0.07} & 1.21 & 1.20 & \gr{0.06}  & 1.17 & 1.16 & \gr{0.06} & \f{1.15} & \gr{0.14} & 1.16 & \gr{6.38}\\
					DIS1 & 4.16 & 5.12 & 4.26 & 4.26 & \gr{0.43} & 5.15 & 5.13 & \gr{0.30}  & 5.12 & 5.12 & \gr{0.59} & 4.19 & \gr{0.14} & \f{4.18} & \gr{8.66}\\
					WEC1 & 3.64 & 7.61 & 5.00 & 4.11 & \gr{1.98} & 7.36 & 7.32 & \gr{1.77}  & 7.32 & 7.31 & \gr{1.95} & \f{4.05} & \gr{0.24} & \f{4.05} & \gr{16.34}\\
					IH & 0.00 & 0.02 & \f{0.00} & \f{0.00} & \gr{57.33} & \f{0.00} & \f{0.00} & \gr{45.02}  & 0.01 & 0.01 & \gr{486.05} & 2.45 & \gr{2.73} & 1.97 & \gr{33.83}\\
					NN14 & 9.43 & 30.10 & 17.53 & 17.49 & \gr{0.15} & 23.29 & 19.90 & \gr{0.12}  & 18.87 & 18.64 & \gr{0.13} & \f{17.48} & \gr{0.16} & \f{17.48} & \gr{7.09}\\
					NN17 & 2.64 & - & - & - & \gr{-} & - & - & \gr{-}  & - & - & \gr{-} & \f{11.22} & \gr{0.05} & \f{11.22} & \gr{0.65}\\
					TDM & 2.12 & 3.16 & 2.70 & \f{2.50} & \gr{0.14} & 3.15 & 3.15 & \gr{0.10}  & 3.10 & 3.10 & \gr{0.11} & - & \gr{-} & 2.57 & \gr{7.13}\\
					DLR1 & 0.06 & 7.82 & 2.79 & 2.79 & \gr{1.14} & 3.10 & 3.02 & \gr{1.03}  & 3.79 & 3.79 & \gr{1.10} & \f{2.78} & \gr{0.07} & \f{2.78} & \gr{1.05}\\
					\bottomrule
			\end{tabular}}
		\end{center}
	\end{table}

	\subsection{A Control Theoretic Interpretation of the Dual Iteration}\label{SHI::sec::interpretation}

	So far the entire dual iteration solely relies on algebraic manipulations by heavily exploiting the elimination lemma \ref{RS::lem::elimination}. This turns an application of the iteration relatively simple but not very insightful. A control theoretic interpretation of the individual steps can be provided based on our robust output-feedback design approach proposed in \cite{HolSch19} that was motivated by the well-known separation principle.
	The classical separation principle states that one can synthesize a stabilizing dynamic output-feedback controller by combining a state observer with a state-feedback controller, which can be designed completely independently from each other. 
	Instead, we proposed in \cite{HolSch19} to design a full-information controller and thereafter to solve a particular robust design problem with a structure that resembles the one in robust estimation. The latter problem is briefly recalled next.
	
	\vspace{1ex}
	
	Suppose that we have synthesized a full-information controller $\t u = F\t y$ via Lemma \ref{SHI::lem::full_info}. Then we can incorporate the gain $F = (F_1, F_2)$ into the closed-loop interconnection \eqref{SHI::eq::cl_of} with the to-be-designed static controller $K$ and some parameter $\del \in [0, 1]$ as depicted on the left in Fig.~\ref{SHI::fig::of_homotop}; here, $\t G$ denotes 
	the open-loop system \eqref{SHI::eq::sys_of} augmented with the virtual measurements $\t y = \smat{x \\ d}$. 
	In this new configuration, we note that the control input $u$ satisfies
	\begin{equation*}
		u = (1 - \del)\t u + \del \h u,
	\end{equation*}
	i.e., it is a convex combination of the outputs of the full-information and of the to-be-designed static output-feedback controller. 
	In particular, for $\del = 0$, we retrieve \eqref{SHI::eq::clF}, the interconnection of the system \eqref{SHI::eq::sys_of} with the full-information controller $u = \t u =  F \t y$ for the output $\t y = \smat{x \\ d}$. On the other hand, for $\del = 1$, we recover the original interconnection \eqref{SHI::eq::cl_of}.
	This motivates to view $\del$ as a homotopy parameter that continuously deforms the prior interconnection into the latter.
	
	\begin{figure}
		\vspace{1ex}
		\begin{center}
			\begin{minipage}{0.49\textwidth}
				\begin{center}
					\includegraphics[]{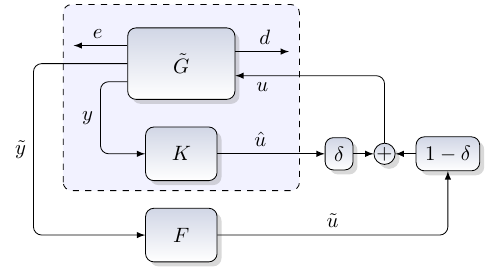}
				\end{center}
			\end{minipage}
			\begin{minipage}{0.49\textwidth}
				\begin{center}
					\includegraphics[]{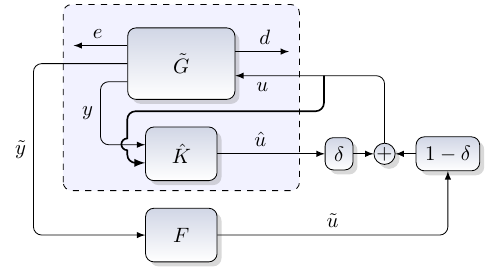}
				\end{center}
			\end{minipage}
		\end{center}
		\caption{Left: Incorporation of the full-information controller gain $F$ into the interconnection \eqref{SHI::eq::cl_of} with the to-be-designed static controller $K$ and some parameter $\del\in [0, 1]$. Right: Allowing the controller to take additional measurements of $u$.}
		\label{SHI::fig::of_homotop}
	\end{figure}
	
	As in \cite{HolSch19} we treat the parameter $\del$ as an uncertainty. A robust design of $K$ turns the achievable upper bounds on the closed-loop energy gain rather conservative. To counteract this conservatism, we allow the to-be-designed controller to additionally include measurements of the convex combination $u$ which results in the configuration on the right in Fig.~\ref{SHI::fig::of_homotop}. This is expected to be beneficial, since the controller knows its own output $\h u$ and, thus, it essentially means to measure the new uncertain signal as well. Note that restricting $\hat K$ to admit the structure $\hat K = (K, 0)$ results again in the configuration on the left in Fig.~\ref{SHI::fig::of_homotop}.
	

	Observe that the control input $u$ can also be expressed as
	\begin{equation*}
		u = (1-\del) \t u + \del \h u = \t u + \del \t z = F_1 x + F_2 d + \t w
		\teq{ with } \t z := \h u - \t u = \h u - F_1 x - F_2d
		\teq{ and } \t w := \del \t z.
	\end{equation*}
	Hence, disconnecting the controller $\hat K$ leads to the uncertain open-loop system
	\begin{equation}
		\mat{c}{\dot x(t) \\ \hline  e(t) \\ \hdashline \t z(t) \\\hdashline \h y(t)} 
		= \mat{c|c:c:c}{A^F & B_1^F & B_2 & 0 \\ \hline
			C_1^F & D_{11}^F & D_{12} & 0 \\ \hdashline
			C_2^F & D_{21}^F & 0 & I \\ \hdashline
			C_3^F & D_{31}^F & D_{32}^F & 0}
		\mat{c}{x(t) \\ \hline d(t) \\ \t w(t) \\ \h u(t)}
		=
		\mat{c|c:c:c}{A +B_2F_1 & B_1 +B_2 F_2 & B_2 & 0\\ \hline
			C_1+D_{12}F_1 & D_{11}+D_{12}F_2 & D_{12}  & 0\\ \hdashline
			-F_1 & -F_2 &  0 & I \\ \hdashline
			C_2 & D_{21} & 0 & 0 \\
			F_1 & F_2 &  I & 0}
		\mat{c}{x(t) \\\hline  d(t) \\ \hdashline \t w(t) \\ \hdashline \h u(t)},  
		\qquad
		\t w(t) = \del \t z(t)
		\label{SHI::eq::sys_es}
	\end{equation}
	for $t\geq 0$ and with the augmented measurements $\h y := \smat{y \\ u}$. Note that the structure of the system \eqref{SHI::eq::sys_es} is closely related to the one appearing in estimation problems as considered, e.g., in \cite{SunPac05, GerDeo01,SchKoe08, Ger99}. As the  essential point, we show that the problem of finding a robust static controller $\hat K$ for \eqref{SHI::eq::sys_es} can be turned convex.
	To this end, note that reconnecting the controller $\hat K$ leads to an uncertain closed-loop system with description
	\begin{equation}
		\mat{c}{\dot x(t) \\ \hline  e(t) \\  \t z(t) } =
		\mat{c|cc}{\Ac^F & \Bc_1^F & \Bc_2^F\\ \hline
			\Cc_1^F & \Dc_{11}^F & \Dc_{12}^F \\ 
			\Cc_2^F & \Dc_{21}^F & \Dc_{22}^F}
		\mat{c}{x(t) \\\hline  d(t) \\ \t w(t) }
		=\mat{c|cc}{A^F & B_1^F & B_2\\ \hline
			C_1^F & D_{11}^F & D_{12} \\ 
			C_2^F +\hat K C_3^F & D_{21}^F + \hat K D_{31}^F & \hat K D_{32}^F}
		\mat{c}{x(t) \\\hline  d(t) \\ \t w(t) },  
		\qquad
		\t w(t) = \del \t z(t).
		\label{SHI::eq::sys_escl}
	\end{equation}
	We analyze the latter system via static IQCs \cite{MegRan97}. Since
	the (uncertain) homotopy parameter $\del$ varies in $[0,1]$, we employ the set of constant multipliers
	\begin{equation*}
		\Pb := \left\{\mat{cc}{0 & H^T \\ H & -H - H^T} ~\middle|~ H + H^T \cg 0 \right\}
		\teq{ because }
		\mat{c}{I \\ \del I}^T P \mat{c}{I \\ \del I}
		=\del (1 - \del) (H + H^T) \cge 0
	\end{equation*}
	holds for all $\del \in [0, 1]$ and any multiplier $P \in \Pb$.
	This leads to the following robust analysis result, which can also be viewed as a special case of the findings in \cite{Sch01}.
	
	\begin{lemma}
		\label{SHI::lem::stabes}
		Let $\Gc_{ij}^F$ be the transfer matrix corresponding to the closed-loop system \eqref{SHI::eq::sys_escl}. Then the uncertain system \eqref{SHI::eq::sys_escl} is well-posed, i.e., $\det(I - \del \Dc_{22}^F) \neq 0$ for all $\del \in [0,1]$, and its energy gain is smaller than $\ga$ for all $\del \in [0, 1]$ if there exist symmetric matrices $X$ and $ P\in \Pb$ satisfying
		\begin{equation}
			X \cg 0
			\teq{ and }
			\Ls\left(\mat{cc}{0 & X \\ X & 0}, \mat{c|c}{P_\ga & 0 \\ \hline 0 & P}, \mat{cc}{\Gc_{11}^F & \Gc_{12}^F \\ I & 0 \\ \hline \Gc_{21}^F & \Gc_{22}^F \\ 0 & I}_\ss \right) \cl 0.
			\label{SHI::lem::lmi_stabes}
		\end{equation}
	\end{lemma}

	The specific structure of the system \eqref{SHI::eq::sys_es} \emph{and} of the multipliers in $\Pb$ allow us to render the design problem on the right in Fig.~\ref{SHI::fig::of_homotop} convex, e.g., by relying on the elimination lemma. This is the first statement of the following result.
	
	\begin{theorem}
		\label{SHI::theo::estimation}
		Let $G_{ij}^F$ denote the transfer matrices corresponding to \eqref{SHI::eq::sys_es} and let $V_F$ be a basis matrix of $\ker\smat{C_2 & D_{21}  & 0 \\ F_1 & F_2 & I}$. Further, suppose that $A^F = A + B_2F_1$ is stable.
		Then there exists a controller $\hat K$ for the system \eqref{SHI::eq::sys_es} such that the robust analysis LMIs \eqref{SHI::lem::lmi_stabes} are satisfied if and only if there exist symmetric matrices $X$ and $ P\in\Pb$ satisfying
		\begin{subequations}
			\label{SHI::theo::eq::lmi_es}
			\begin{equation}
				\arraycolsep=2pt
				%
				V_F^T\Ls \left( \mat{cc}{0 & X \\ X & 0}, \mat{c:c}{P_\ga & 0  \\ \hdashline 0 & P},
				\mat{cc}{G_{11}^F & G_{12}^F \\ I & 0 \\ \hdashline 
					G_{21}^F & G_{22}^F  \\ 0 & I }_{\ss}\right)V_F \cl 0
				\teq{ and }
				\Ls\left(\mat{cc}{0 & X \\ X & 0},  P_\ga,
				\mat{c}{G_{11}^F  \\ I }_\ss \right) \cl 0.
				\dlabel{SHI::theo::eq::lmi_esa}{SHI::theo::eq::lmi_esb}
			\end{equation}
		\end{subequations}
		Moreover, the above LMIs are feasible if and only if the primal synthesis LMIs \eqref{SHI::theo::eq::lmi_ofF} are feasible. In particular, feasibility of \eqref{SHI::theo::eq::lmi_es} implies that there exists a controller $K$ for the original system \eqref{SHI::eq::sys_of} such that the analysis LMIs \eqref{SHI::lem::lmi_stab} are feasible for the corresponding closed-loop interconnection.
	\end{theorem}
	\vspace{1ex}

	In \cite{HolSch19} we gave a trajectory based proof of the second statement of this result in the context of robust output-feedback design. Here, we show solely based on algebraic manipulations and on the LFR framework that 
	Theorem \ref{SHI::theo::estimation} actually recovers the primal design result Theorem \ref{SHI::theo::ofF} while having a nice interpretation in terms of Fig.~\ref{SHI::fig::of_homotop}.
	
	\begin{proof}
		{\it First statement:} Suppose that there is a controller $\hat K$ such that the closed-loop robust analysis LMIs \eqref{SHI::lem::lmi_stabes} are satisfied. Further, note that $U = \smat{I & 0 & 0 \\ 0 & I & 0}^T$ is an annihilator for $(0, 0, I)$ and that $P^{-1} = \smat{\bullet & \bullet \\ \bullet & 0}$ due to the structure of multipliers in $\Pb$. Employing the elimination lemma then leads to the LMI \eqref{SHI::theo::eq::lmi_esa} and to
		\begin{equation*}
			\arraycolsep=0.5pt
			0 \cl (\bullet)^T \mat{cc|c:c}{0 & X & 0 & 0 \\ X & 0 & 0 & 0 \\ \hline 0 & 0 & P_\ga & 0 \\ \hdashline 0 & 0 & 0 & P}^{-1}
			\mat{ccc}{I & 0 & 0 \\ -(A^F)^T & -(C_1^F)^T & -(C_2^F)^T \\ \hline 
				0 & I & 0 \\ -(B_1^F)^T & -(D_{11}^F)^T & -(D_{21}^F)^T \\ \hdashline 
				0 & 0 & I \\ -B_2^T & -D_{12}^T & 0} U 
			%
			%
			= (\bullet)^T \mat{cc|c}{0 & X^{-1} & 0 \\ X^{-1} & 0 & 0 \\ \hline 0 & 0 & P_\ga^{-1} }
			\mat{cc}{I & 0 \\ -(A^F)^T & -(C_1^F)^T  \\ \hline 
				0 & I  \\ -(B_1^F)^T & -(D_{11}^F)^T}.
		\end{equation*}
		An application of the dualization lemma \ref{RS::lem::dualization} yields \eqref{SHI::theo::eq::lmi_esb} and finishes the necessity part of the proof. The converse is obtained by reversing the arguments.

		{\it Second statement:} Observe that a valid annihilator $V_F$ is given by the choice $V_F = \smat{I \\ -F} V$ with $V$ being a basis matrix of $\ker(C_2, D_{21})$. Moreover, via elementary computations and by recalling \eqref{SHI::eq::sys_es}, we have
		\begin{equation*}
			\arraycolsep=1pt
			\mat{c:c}{G_{11}^F & G_{12}^F \\ I & 0 \\ \hdashline 
				G_{21}^F & G_{22}^F \\ 0 & I }_{\ss}\! \mat{c}{I \\ -F}\!
			%
			%
			= \mat{c}{\mat{c}{G_{11}  \\ I 
				}_{\ss} \\\hdashline - \mat{c}{I \\ I}F}
			\text{ and thus \eqref{SHI::theo::eq::lmi_esa} reads as }
			0 \cg V^T\Ls\left(\mat{cc}{0 & X \\ X & 0},  P_\ga, \mat{c}{G_{11} \\ I 
			}_{\ss} \right)  V 
			+ (\bullet)^T \mat{c}{I \\ I}^T P \mat{c}{I \\ I}FV.
		\end{equation*}
		The latter inequality is actually identical to \eqref{SHI::theo::eq::lmi_ofFb} since $\smat{I \\ I}^TP\smat{I \\ I} = 0$ follows from $P \in \Pb$. This shows that feasibility of \eqref{SHI::theo::eq::lmi_es} implies validity of \eqref{SHI::theo::eq::lmi_ofF}. Conversely, if \eqref{SHI::theo::eq::lmi_ofFb} is satisfied, we can pick any $P \in \Pb$ and infer that the latter inequality is true, which leads to \eqref{SHI::theo::eq::lmi_es}.
	\end{proof}

	As the most important benefit of the above interpretation, the design problem corresponding to Fig.~\ref{SHI::fig::of_homotop} can alternatively be solved, e.g., via a convexifying parameter transformation instead of elimination and in various other important scenarios. In particular, this allows for an extension of the dual iteration to situations where elimination is not or only partly possible. 
	To this end, let us show how to solve the design problem corresponding to Fig.~\ref{SHI::fig::of_homotop} without elimination.
	
	\begin{theorem}
		\label{SHI::theo::ofF_par}
		Suppose that $A^F = A + B_2F_1$ is stable. Then there exists a controller $\hat K$ for the system \eqref{SHI::eq::sys_es} such that the robust analysis LMIs \eqref{SHI::lem::lmi_stabes} are feasible if and only if there exists matrices $H$, $N = (N_1, N_2)$ and a symmetric matrix $X$ satisfying
		\begin{subequations}
			\label{SHI::theo::eq::lmi_ofF_par}
			\begin{equation}
				%
				H + H^T \cg 0
				\teq{ and }
				\Ls\left(\mat{cc}{0  & X \\ X  & 0},
				\mat{c:cc}{P_\ga & 0 & 0\\ \hdashline 0 & 0 & I \\ 0 & I & -H-H^T},
				\mat{ccc}{A^F  & B_1^F  & B_2 \\ \hline
					C_1^F  & D_{11}^F  & D_{12}  \\ 0 & I & 0 \\ \hdashline
					HC_2^F + NC_3^F & HD_{21}^F + ND_{31}^F & ND_{32}^F  \\ 0 & 0 & I} \right) \cl 0.
				\dlabel{SHI::theo::eq::lmi_ofF_para}{SHI::theo::eq::lmi_ofF_parb}
			\end{equation}
		\end{subequations}
		%
		%
		If the above LMIs are feasible, a static controller $K$ such that the analysis LMIs \eqref{SHI::lem::lmi_stab} are feasible for the closed-loop system \eqref{SHI::eq::cl_of} is given by $K := (H - N_2)^{-1}N_1$.
	\end{theorem}

	\begin{proof}
		We only prove the sufficiency part of the first statement and the second statement for brevity.
		Note at first that $X\cg 0$ follows from stability of $A^F$ and by considering the left upper block of \eqref{SHI::theo::eq::lmi_ofF_parb}.
		Moreover, observe that $H$ is nonsingular by $H + H^T \cg 0$. 
		%
		Then we can rewrite \eqref{SHI::theo::eq::lmi_ofF_parb} with $P := \smat{0 & H^T  \\ H  & -H-H^T} \in \Pb$ and $\hat K := H^{-1}N$ as
		\begin{equation}
			0 \cg \Ls\left(\mat{cc}{0 & X  \\ X  & 0},
			\mat{c:c}{P_\ga & 0 \\ \hdashline 0 & P},
			\mat{ccc}{A^F & B_1^F  & B_2 \\ \hline
				C_1^F  & D_{11}^F  & D_{12}  \\ 0 & I & 0 \\ \hdashline
				C_2^F + \hat K C_3^F  & D_{21}^F + \hat K D_{31}^F  & \hat K D_{32}^F  \\ 0 & 0 & I} \right)
			= \Ls\left(\mat{cc}{0 & X \\ X & 0}, \mat{c|c}{P_\ga & 0 \\ \hline 0 & P}, \mat{cc}{\Gc_{11}^F & \Gc_{12}^F \\ I & 0 \\ \hline \Gc_{21}^F & \Gc_{22}^F \\ 0 & I}_\ss \right).
			\label{SHI::pro::lmiF}
		\end{equation}
		In particular, $\hat K$ is a controller for the system \eqref{SHI::eq::sys_es} as desired. Moreover, from the right lower block of \eqref{SHI::pro::lmiF} and the structure of $P$ we infer
		\begin{equation*}
			\mat{c}{\hat K D_{32}^F \\ I }^TP\mat{c}{\hat K D_{32}^F \\ I } \cl 0
			\teq{ and }
			\mat{c}{I \\ \del I}^T P \mat{c}{I \\ \del I} \cge 0
			\teq{ for all }\del \in [0, 1].
		\end{equation*}
		This implies $\det(I - \del \hat K D_{32}^F) \neq 0$ for all $\del \in [0, 1]$ and, in particular, that $I - \hat K D_{32}^F$ is nonsingular. By recalling the abbreviations in \eqref{SHI::eq::sys_es}, we get
		\begin{equation*}
			W := (W_1, W_2) := (I - \hat K D_{32}^F)^{-1}(C_2^F + \hat K C_3^F,~ D_{21}^F + \hat K D_{31}^F) = -F + (H - N_2)^{-1}N_1 (C_2, D_{21}) = -F + K(C_2, D_{21}).
		\end{equation*}
		Then we obtain via elementary computations that 
		\begin{equation*}
			\arraycolsep=3pt
			\mat{ccc}{I & 0 & 0 \\ A^F  & B_1^F  & B_2 \\ \hline
				C_1^F  & D_{11}^F  & D_{12}  \\ 0 & I & 0 \\ \hdashline
				C_2^F + \hat K C_3^F  & D_{21}^F + \hat K D_{31}^F  & \hat K D_{32}^F  \\ 0 & 0 & I}
			\mat{cc}{I & 0 \\ 0 & I \\ 
				W_1  & W_2 }
			= \mat{cc}{I & 0 \\ 
				A + B_2 K C_2 & B_1 + B_2 K D_{21} \\ \hline
				C_1 + D_{12}K C_2 & D_{11} + D_{12}K D_{21} \\ 0 & I \\ \hdashline
				W_1  & W_2  \\ W_1 & W_2 }.
		\end{equation*}
		In particular, we can infer from \eqref{SHI::pro::lmiF} that 
		\begin{equation*}
			\Ls\left(\mat{cc}{0 & X  \\ X  & 0}, P_\ga,
			\mat{cc}{A + B_2KC_2 & B_1 + B_2KD_{21} \\ \hline
				C_1 + D_{12}KC_2 & D_{11} + D_{12}KD_{21} \\ 0 & I }\right)
			\cl - W^T\mat{c}{I \\ I}^T P \mat{c}{I \\ I} W = 0
		\end{equation*}
		holds, which yields the last claim.	
	\end{proof}

	As it has been discussed for the primal one, the dual design result in Theorem \ref{SHI::theo::ofE} can as well be interpreted as the solution to the dual synthesis problem corresponding to Fig.~\ref{SHI::fig::of_homotop}. This is closely related to a feedforward synthesis problem. 
	In fact, it can be viewed as a separation-like result which involves the consecutive construction of a full-actuation controller and a corresponding feedforward-like controller.
	To be concrete, for a given full-actuation gain $E = (E_1^T, E_2^T)^T$, the dual synthesis problem corresponding to Fig.~\ref{SHI::fig::of_homotop} amounts to finding a static controller $\hat K$ such that the robust analysis LMIs \eqref{SHI::lem::lmi_stabes} are feasible for the interconnection of the controller $\hat K$ and the uncertain open-loop system
	\begin{equation}
		\mat{c}{\dot x(t) \\ \hline e(t) \\ \hdashline \t z(t) \\ \hdashline \h y(t)}
		= \mat{c|c:c:c}{A^E & B_1^E & B_2^E & B_3^E \\ \hline
			C_1^E & D_{11}^E & D_{12}^E & D_{13}^E \\ \hdashline
			C_2 & D_{21} & 0 & D_{23}^E \\ \hdashline
			0 & 0 & I & 0}
		\mat{c}{x(t) \\ \hline d(t) \\ \hdashline \t w(t) \\ \hdashline \h u(t)}
		= \mat{c|c:c:cc}{A + E_1C_2 & B_1 + E_1D_{21} & -E_1 & B_2 & E_1 \\ \hline
			C_1 + E_2C_2 & D_{11} + E_2D_{21} & -E_2 & D_{12} & E_2\\ \hdashline
			C_2 & D_{21} & 0 & 0 & I \\ \hdashline
			0 & 0 & I & 0  & 0}
		\mat{c}{x(t) \\ \hline d(t) \\ \hdashline\t w(t) \\ \hdashline \h u(t)},\qquad
		\t w(t) = \del \t z(t).
		\label{SHI::eq::clEff}
	\end{equation}
	A convex solution to this design problem is given by the following result. Apart from an application of the dualization lemma~\ref{RS::lem::dualization}, the proof is almost identical to the one of Theorem \ref{SHI::theo::ofF_par} and thus omitted for brevity.
	
	\begin{theorem}
		\label{SHI::theo::ofE_par}
		Suppose that $A^E = A + E_1C_2$ is stable. Then there exists a controller $\hat K$ for the system \eqref{SHI::eq::clEff} such that the LMIs \eqref{SHI::lem::lmi_stabes} are feasible for the resulting closed-loop system if and only if there exists matrices $H$, $N = (N_1^T, N_2^T)^T$ and a symmetric matrix $Y$ satisfying
		\begin{subequations}
			\label{SHI::theo::eq::lmi_ofE_par}
			\begin{equation}
				%
				H + H^T \cg 0
				\teq{ and }
				\Ls\left(\mat{cc}{0 & Y \\ Y  & 0},
				\mat{c:cc}{P_\ga^{-1} & 0 & 0\\ \hdashline 0 & H+ H^T & I \\ 0 & I & 0},
				\mat{ccc}{-(A^E)^T  & -(C_1^E)^T  & -C_2^T \\[0.2ex] \hline
					0 & I & 0 \\ -(B_1^E)^T  & -(D_{11}^E)^T  & -D_{21}^T   \\[0.2ex] \hdashline
					0 & 0 & I \\ -(B_2^EH + B_3^E N)^T &  -(D_{12}^EH + D_{13}^E N)^T & -(D_{23}^EN)^T} \right) \cg 0.
				\dlabel{SHI::theo::eq::lmi_ofE_para}{SHI::theo::eq::lmi_ofE_parb}
			\end{equation}
		\end{subequations}
		If the above LMIs are feasible, a static controller $K$ such that the analysis LMIs \eqref{SHI::lem::lmi_stab} are feasible for the closed-loop system \eqref{SHI::eq::cl_of} is given by
		$K := N_1(H - N_2)^{-1}$.
	\end{theorem}

	Let us conclude the section with an interesting observation. Due to the elimination lemma, feasibility of the primal synthesis LMIs \eqref{SHI::theo::eq::lmi_ofF} is equivalent to the existence of a static output-feedback controller $K$ and a \emph{common} certificate $X$ satisfying
	\begin{equation}
		%
		\Ls\left(\mat{cc}{0 & X \\ X & 0},P_\ga , \mat{c}{G_{11}^F \\ I}_\ss \right) \cl 0
		\teq{ and }
		\Ls\left(\mat{cc}{0 & X \\ X & 0}, P_\ga , \mat{c}{\Gc \\ I}_\ss \right) \cl 0.
		\label{SHI::eq::lmiFcond}
	\end{equation}
	for $G^F_{11}$ as in Theorem \ref{SHI::theo::ofF} and for $\Gc = [\Ac, \Bc, \Cc, \Dc]$ being the transfer matrix corresponding to \eqref{SHI::eq::cl_of} the closed-loop interconnection of the system \eqref{SHI::eq::sys_of} and the controller $K$. 
	Thus, by solving the primal synthesis LMIs, the dual iteration aims in each primal step to find a static controller $K$, which is linked to the given full-information controller $F$ through the common certificate $X$. 
	This shows once more that the suggested initialization in Remark~\ref{SHI::rema::init} makes sense for the dual iteration as well.

	Due to Theorem \ref{SHI::theo::estimation}, we also know that feasibility of the primal synthesis LMIs \eqref{SHI::theo::eq::lmi_ofF} is equivalent to the existence of a controller $\hat K$ such that the robust analysis LMIs \eqref{SHI::lem::lmi_stabes} are satisfied for the closed-loop system \eqref{SHI::eq::sys_escl}. Let us provide some alternative arguments that the existence of such a controller $\hat K$ is equivalent to feasibility of the LMIs \eqref{SHI::eq::lmiFcond}:

	Let a suitable controller $\hat K = (K_1, K_2)$ be given. Then note that the uncertain closed-loop system \eqref{SHI::eq::sys_escl} can also be expressed as
	\begin{equation*}
		\arraycolsep=4pt
		\mat{c}{\dot x(t) \\ e(t)}
		=\mat{cc}{A + B_2 (I - K_2 \del)^{-1}\big[(1-\del)F_1 + \del K_1 C_2 \big] & 
			B_1 + B_2(I - K_2 \del)^{-1} \big[(1-\del)F_2 + \del K_1 D_{21} \big] \\
			C_1 + D_{12}(I - K_2 \del)^{-1} \big[(1-\del)F_1 + \del K_1 C_2 \big] &
			D_{11} + D_{12}(I - K_2\del )^{-1} \big[(1-\del)F_2 + \del K_2 D_{21} \big]}
		\mat{c}{x(t) \\ d(t)};
	\end{equation*}
	in the sequel we abbreviate the above system matrices as $A(\del)$, $B(\del)$, $C(\del)$ and $D(\del)$, respectively. Since the robust analysis LMIs \eqref{SHI::lem::lmi_stabes} are satisfied, we infer, in particular, that
	\begin{equation}
		\Ls\left(\mat{cc}{0 & X \\ X & 0},P_\ga , \mat{cc}{A(\del) & B(\del)\\ \hline C(\del) & D(\del) \\ 0 & I}\right) \cl 0
		\teq{ for all }\del \in [0, 1].
		\label{SHI::eq::lmiFcond2}
	\end{equation}
	This yields \eqref{SHI::eq::lmiFcond} for $K:= (I - K_2)^{-1}K_1$ by considering the special cases $\del = 0$ and $\del = 1$.
	
	Conversely, suppose that \eqref{SHI::eq::lmiFcond} holds for some static gain $K$. Then we can apply the Schur complement twice to infer
	\begin{equation*}
		\mat{ccc}{(A^F)^TX\! +\! XA^F & XB_1^F & (\bullet)^T \\
			(\bullet)^T & - \ga^2 I & (\bullet)^T \\ C_1^F & D_{11}^F & - I} \cl 0
		\text{ ~as well as~ }
		\mat{ccc}{\Ac^TX \!+\! X\Ac & X\Bc & (\bullet)^T \\
			(\bullet)^T & - \ga^2 I & (\bullet)^T \\ \Cc & \Dc & - I} \cl 0
		\text{ ~and thus~ }
		\mat{ccc}{A(\del)^TX \!+\! XA(\del) & XB(\del) & (\bullet)^T \\
			(\bullet)^T & - \ga^2 I & (\bullet)^T \\ C(\del) & D(\del) & - I} \cl 0
	\end{equation*}
	for all $\del \in [0, 1]$ and for $\hat K = (K, 0)$ by convexity. Applying the Schur complement once more yields again \eqref{SHI::eq::lmiFcond2}. 
	From the full block S-procedure\cite{Sch97}, we infer the existence of a symmetric matrix $\t P$ such that the LMIs \eqref{SHI::lem::lmi_stabes} and $\smat{I \\ \del I}^T \t P \smat{I \\ \del I} \cge 0$ hold for all $\del \in [0, 1]$. As argued in \cite{DetSch01} it is finally possible to find some $P \in \Pb$ satisfying \eqref{SHI::lem::lmi_stabes} as well. 

	\section{Static Output-Feedback Multi-Objective Design}\label{SMO::sec::smo}
	
	In this section, we consider the synthesis of static output-feedback controllers satisfying multiple design specifications. As elaborated on, e.g., in \cite{Sch00a, SchGah97, EbiPea15, GumHen09}, such synthesis problems with multiple objectives are more challenging than those with a single objective; in particular, the corresponding non-convex design of static controllers is rendered even more difficult. 
	Multi-objective design problems are particularly challenging in the context of the dual iteration because the elimination lemma \ref{RS::lem::elimination} is no longer applicable. We rely on the interpretation and results in Section \ref{SHI::sec::interpretation} in order to provide a novel variant of the dual iteration.

	\subsection{Problem Description}

	For fixed real matrices of appropriate dimensions and initial conditions $x(0) \in \R^n$, we now consider the open-loop system
	\begin{equation}
		\arraycolsep=1pt
		\mat{c}{\dot x(t) \\\hline  e_1(t) \\ e_2(t)\\ y(t)} 
		= \mat{c|ccc}{A & B_1 & B_2 & B_3 \\ \hline C_1 & D_{11} & D_{12} & D_{13} \\ C_2 & D_{21} & D_{22} & D_{23} \\ C_3 & D_{31} & D_{32} & 0} 
		\mat{c}{x(t) \\\hline  d_1(t) \\ d_2(t) \\ u(t)}
		\label{SMO::eq::sys_of}
	\end{equation}
	for $t\geq 0$.
	For a given symmetric matrix $P= \smat{Q & S \\ S^T & R}$ with $R\cge 0$, we aim in this section to design a static controller
	\begin{equation}
		u(t) = K y(t)
		\label{SMO::eq::con_of}
	\end{equation}
	for the system  \eqref{SMO::eq::sys_of} such that the $H_\infty$-norm $\|\Gc_{11}\|_\infty = \sup_{\omega \in \R}\|\Gc_{11}(i\omega)\|$ is as small as possible and such that, additionally,
	\begin{equation}
		\mat{c}{\Gc_{22}(i\omega) \\ I}^\ast P\mat{c}{\Gc_{22}(i\omega) \\ I} \cl 0
		\teq{ for all }\omega \in \R \cup \{\infty \}
		\label{SMO::eq::scnd_constr}
	\end{equation}
	is satisfied; here $\Gc_{jj} := [A + B_3KC_3, B_j+B_3KD_{3j}, C_j + D_{j3}KC_3, D_{jj}+D_{j3}KD_{3j}]$ denote the closed-loop transfer matrices corresponding to the channel from $d_j$ to $e_j$ for $j \in \{1,2\}$. 
	The second objective \eqref{SMO::eq::scnd_constr}, characterized by the matrix $P$, can be employed, e.g., to enforce mandatory gain or passivity constraints on the closed-loop interconnection of \eqref{SMO::eq::sys_of} and \eqref{SMO::eq::con_of}. 
	This additional objective turns the problem of finding a suitable controller into a difficult static multi-objective design problem. Of course, one can include more than two objectives and it is also possible to deal, e.g., with constraints on the $H_2$-norm. However, we focus on the above setup for didactic reasons. Our approach is based on the following result which is a well-known and minor extension of the bounded real lemma.
	
	\begin{lemma}
		\label{SMO::theo::of}
		Let $P_\gamma := \smat{I & 0 \\ 0 & -\ga^2 I}$ and let $\Gc_{ij}$ denote the transfer matrices corresponding to the interconnection of \eqref{SMO::eq::sys_of} and some controller \eqref{SMO::eq::con_of}. Then $\|\Gc_{11}\|_\infty < \ga$ and \eqref{SMO::eq::scnd_constr} hold if and only if there exist positive definite matrices $X$ and $Y$ satisfying 
		\begin{equation}
			\Ls\left(\mat{cc}{0 & X \\ X & 0}, P_\ga, \mat{c}{\Gc_{11} \\ I
			}_{\ss} \right)  \cl 0
			\teq{ and }
			\Ls\left(\mat{cc}{0 & Y \\ Y & 0}, P, \mat{c}{\Gc_{22} \\ I}_{\ss} \right)  \cl 0.
			\label{SMO::theo::eq::lmi_ana}
		\end{equation}
		We denote by $\ga_\opt$ the infimal $\ga > 0$ such that there exists a controller \eqref{SMO::eq::con_of} that renders the analysis LMIs \eqref{SMO::theo::eq::lmi_ana} feasible for the corresponding closed-loop system.
	\end{lemma}

	\begin{remark}
		Similarly as in the previous section, informative lower bounds on $\ga_\opt$ can be obtained by considering the corresponding dynamic output-feedback multi-objective design problems. 
		The latter problem admits a convex solution via the Youla parametrization \cite{Fra87, Sch00a} which can be costly to compute. A cheaper but also less informative lower bound is obtained by performing a standard $H_\infty$-design involving a dynamic controller for the system \eqref{SMO::eq::sys_of} without the channel from $d_2$ to $e_2$. 
	\end{remark}

	\subsection{Dual Iteration}
	
	We present now a variant of the dual iteration in order to compute upper bounds on $\ga_\opt$ and corresponding static controllers \eqref{SMO::eq::con_of}. We begin by stating the primal design result which is motivated by Theorem \ref{SHI::theo::ofF_par} and now involves two full-information gains $F_1 = (F_{11}, F_{12})$ and $F_2 = (F_{21}, F_{22})$, one for each of the objectives respectively.  
	The proof is almost identical to the one of Theorem \ref{SHI::theo::ofF_par} and thus omitted for brevity.
	
	\begin{theorem}
		\label{SMO::theo::ofF}
		There exists a controller \eqref{SMO::eq::con_of} such that the closed-loop analysis LMIs \eqref{SMO::theo::eq::lmi_ana} are feasible if there exists matrices $H$, $N = (N_1, N_2)$ and symmetric matrices $X$, $Y$ satisfying
		\begin{subequations}
			\label{SMO::theo::eq::lmi_ofF}
			\begin{equation}
				X\cg 0, \quad Y \cg 0,\quad H+H^T \cg 0,
			\end{equation}
			\begin{equation}
				\Ls\left(\mat{cc}{0  & X \\ X  & 0},
				\mat{c:cc}{P_\ga & 0 & 0\\ \hdashline 0 & 0 & I \\ 0 & I & -H-H^T},
				\mat{ccc}{A + B_3F_{11}  & B_1 + B_3F_{12}  & B_3 \\ \hline
					C_1 +D_{13}F_{11}  & D_{11}+D_{13}F_{12}  & D_{13}  \\ 0 & I & 0 \\ \hdashline
					(N_2-H)F_{11} + N_1C_3 & (N_2 - H)F_{12} + N_1D_{31} & N_2  \\ 0 & 0 & I} \right) \cl 0
				\label{SMO::theo::eq::lmi_ofFd}
			\end{equation}
			and
			\begin{equation}
				\Ls\left(\mat{cc}{0  & Y \\ Y  & 0},
				\mat{c:cc}{P & 0 & 0\\ \hdashline 0 & 0 & I \\ 0 & I & -H-H^T},
				\mat{ccc}{A + B_3F_{21}  & B_2 + B_3F_{22}  & B_3 \\ \hline
					C_2 +D_{23}F_{21}  & D_{22}+D_{23}F_{22}  & D_{23}  \\ 0 & I & 0 \\ \hdashline
					(N_2-H)F_{21} + N_1C_3 & (N_2 - H)F_{22} + N_1D_{32} & N_2  \\ 0 & 0 & I} \right) \cl 0.
				\label{SMO::theo::eq::lmi_ofFe}
			\end{equation}
		\end{subequations}
		If the above LMIs are feasible, a static gain $K$ such that the analysis LMIs \eqref{SMO::theo::eq::lmi_ana} are feasible is given by 
		$K := (H - N_2)^{-1}N_1$. Moreover, we have $\ga_\opt \leq \ga_F$ for $\ga_F$ being the infimal $\ga > 0$ such that the above LMIs are feasible.
	\end{theorem}
	
	Observe that we employ identical matrices $H$ and $N$ in the LMIs \eqref{SMO::theo::eq::lmi_ofFd} and \eqref{SMO::theo::eq::lmi_ofFe} corresponding to the two different objectives.  
	In contrast to many other LMI based approaches, this choice does not introduce any conservatism in the following sense.
	
	\begin{theorem}
		\label{SMO::theo::necess}
		There exists a controller \eqref{SMO::eq::con_of} such that the closed-loop analysis LMIs \eqref{SMO::theo::eq::lmi_ana} are feasible if and only if there exist full-information gains $F_1, F_2$, matrices $H, N$ and symmetric matrices $X,Y$ satisfying \eqref{SMO::theo::eq::lmi_ofF}.
	\end{theorem}
	
	\begin{proof}
		Sufficiency follows from Theorem \ref{SMO::theo::ofF} and it remains to show necessity. To this end, let $X, Y$ and $K$ be matrices satisfying the inequalities \eqref{SMO::theo::eq::lmi_ana}. With those matrices as well as
		$F_1 := K(C_3, D_{31})$, $F_2 := K(C_3, D_{32})$ and $N:= H(K, 0)$ for some to-be-chosen matrix $H$, the left hand side of \eqref{SMO::theo::eq::lmi_ofFd} equals
		\begin{equation*}
			\Ls\left(\mat{cc}{0  & X \\ X  & 0},
			\mat{c:cc}{P_\ga & 0 & 0\\ \hdashline 0 & 0 & I \\ 0 & I & -H-H^T},
			\mat{ccc}{A + B_3KC_3  & B_1 + B_3KD_{31}  & B_3 \\ \hline
				C_1 +D_{13}KC_3  & D_{11}+D_{13}KD_{31}  & D_{13}  \\ 0 & I & 0 \\ \hdashline
				0 & 0 & 0  \\ 0 & 0 & I} \right)
			= \mat{cc}{\Gamma_{11} & \Gamma_{12} \\ \Gamma_{12}^T & \Gamma_{22}} - \mat{cc}{0 & 0 \\ 0 & H+H^T}.
		\end{equation*}
		Here, the blocks $\Gamma_{ij}$ do not depend on $H$ and $\Gamma_{11}$ is identical to the left-hand side of the first LMI in \eqref{SMO::theo::eq::lmi_ana}. Thus $\Gamma_{11}$ is negative definite and we can hence infer that \eqref{SMO::theo::eq::lmi_ofFd} is satisfied for $H = \alpha I$ and some large enough $\alpha >0$. Finally, we can argue analogously and increase $\alpha$ if necessary to conclude that  \eqref{SMO::theo::eq::lmi_ofFe} is satisfied as well.
	\end{proof}

	Analogously, the corresponding dual design result is motivated by Theorem \ref{SHI::theo::ofE_par} and involves two full-actuation gains $E_1 = (E_{11}^T, E_{12}^T)^T$ and $E_2 = (E_{21}^T, E_{22}^T)^T$.

	\begin{theorem}
		\label{SMO::theo::ofE}
		There exists a controller \eqref{SMO::eq::con_of} such that the closed-loop analysis LMIs \eqref{SMO::theo::eq::lmi_ana} are feasible if there exists matrices $H$, $N = (N_1^T, N_2^T)^T$ and symmetric matrices $X$, $Y$ satisfying
		\begin{subequations}
			\label{SMO::theo::eq::lmi_ofE}
			\begin{equation}
				X \cg 0,\quad Y \cg 0,\quad H + H^T \cg 0,
			\end{equation}
			\begin{equation}
				\Ls\left(\mat{cc}{0 & X \\ X  & 0},
				\mat{c:cc}{P_\ga^{-1} & 0 & 0\\ \hdashline 0 & H+ H^T & I \\ 0 & I & 0},
				\mat{ccc}{-(A + E_{11}C_3)^T  & -(C_1+E_{12}C_3)^T  & -C_3^T \\[0.2ex] \hline
					0 & I & 0 \\ -(B_1+E_{11}D_{31})^T  & -(D_{11}+E_{12}D_{31})^T  & -D_{31}^T   \\[0.2ex] \hdashline
					0 & 0 & I \\ -\big(E_{11}(N_2 - H) + B_3N_1\big)^T &  -\big(E_{12}(N_2 -H) + D_{13}N_1\big)^T & -N_2^T} \right) \cg 0
				\label{SMO::theo::eq::lmi_ofEd}
			\end{equation}
			and
			\begin{equation}
				\Ls\left(\mat{cc}{0 & Y \\ Y  & 0},
				\mat{c:cc}{P^{-1} & 0 & 0\\ \hdashline 0 & H+ H^T & I \\ 0 & I & 0},
				\mat{ccc}{-(A + E_{21}C_3)^T  & -(C_2+E_{22}C_3)^T  & -C_3^T \\[0.2ex] \hline
					0 & I & 0 \\ -(B_2+E_{21}D_{32})^T  & -(D_{22}+E_{22}D_{32})^T  & -D_{32}^T   \\[0.2ex] \hdashline
					0 & 0 & I \\ -\big(E_{21}(N_2 - H) + B_3N_1\big)^T &  -\big(E_{22}(N_2 -H) + D_{23}N_1\big)^T & -N_2^T} \right) \cg 0.
				\label{SMO::theo::eq::lmi_ofEe}
			\end{equation}
		\end{subequations}
		If the above LMIs are feasible, a static gain $K$ such that the analysis LMIs \eqref{SMO::theo::eq::lmi_ana} are feasible is given by $K := N_1(H - N_2)^{-1}$.  Moreover, we have $\ga_\opt \leq \ga_E$ for $\ga_E$ being the infimal $\ga > 0$ such that the above LMIs are feasible.
	\end{theorem}

	Similarly as stated in Theorem \ref{SHI::theo::it_summary}, the primal and dual design results Theorems \ref{SMO::theo::ofF} and \ref{SMO::theo::ofE} can be sequentially applied. Indeed, suppose that the LMIs \eqref{SMO::theo::eq::lmi_ofE} are feasible for some $\ga > 0$. Then Theorem \ref{SMO::theo::ofE} implies the existence of a static gain $K$ such that the analysis LMIs \eqref{SMO::theo::eq::lmi_ana} are feasible for $\ga$. Due to Theorem \ref{SMO::theo::necess} we can find some full-information gains $F_1$ and $F_2$ such that the LMIs \eqref{SMO::theo::eq::lmi_ofF} are feasible for exactly the same $\ga > 0$. Note that superior full-information gains than the ones proposed in the proof of Theorem \ref{SMO::theo::necess} can be obtained, e.g., by minimizing $\ga$ subject to \eqref{SMO::theo::eq::lmi_ofF} with variables $X, Y, H, F_1, F_1, \ga$ and for $N:= H(K, 0)$. 
	
	Analogously, we infer from the feasibility of the primal synthesis LMIs \eqref{SMO::theo::eq::lmi_ofF} the existence of full-actuation gains such that the dual synthesis LMIs \eqref{SMO::theo::eq::lmi_ofE} are feasible as well.
	In particular, the following dual iteration generates a monotonically decreasing sequence $(\ga^k)_{k \in \N}$ of upper bounds on $\ga_\opt$.
	
	\begin{Algorithm}
		\label{SMO::algo::dual_iteration}
		Dual iteration for static output-feedback multi-objective synthesis.
		\begin{enumerate}
			\item \emph{Initialization:} Set $\ga^0 := +\infty$, $k = 1$ and design initial full-information gains $F_1$ and $F_2$.
			\item \emph{Primal step:} Compute $\ga_F$ by solving the primal synthesis LMIs \eqref{SMO::theo::eq::lmi_ofF} for the given gains $F_1$, $F_2$ and choose some small $\eps_k>0$ such that $\ga^k := \ga_F(1+\eps_k) < \ga^{k-1}$. For $\ga = \ga^k$, determine matrices $X$, $Y$ satisfying the LMIs \eqref{SMO::theo::eq::lmi_ofF} and construct full-actuation gains $E_1$, $E_2$ satisfying the dual synthesis LMIs \eqref{SMO::theo::eq::lmi_ofE} for $(X, Y)$ replaced by $(X^{-1}, Y^{-1})$. 
			\item \emph{Dual step:} Compute $\ga_E$ by solving the dual synthesis LMIs \eqref{SMO::theo::eq::lmi_ofE} for the given gains $E_1$, $E_2$ and choose some small $\eps_{k+1}> 0$ such that $\ga^{k+1} := \ga_E (1+\eps_{k+1}) < \ga^k$. For $\ga = \ga^{k+1}$, determine matrices $X$, $Y$ satisfying the LMIs \eqref{SMO::theo::eq::lmi_ofE} and construct full-information gains $F_1$, $F_2$ satisfying the primal synthesis LMIs \eqref{SMO::theo::eq::lmi_ofF} for $(X, Y)$ replaced by $(X^{-1}, Y^{-1})$.
			\item \emph{Termination:} If $k$ is too large or $\ga^k$ does not decrease any more, then stop and construct a static output-feedback controller according to Theorem \ref{SMO::theo::ofE}.\\
			Otherwise set $k = k+2$ and go to the primal step. 
		\end{enumerate}
	\end{Algorithm}

	\begin{remark}
		Algorithm \ref{SMO::algo::dual_iteration} can be initialized, e.g., by considering the synthesis LMIs corresponding to the design of so-called mixed controllers as proposed, e.g., in \cite{SchGah97}. These are dynamic controllers and their design relies on employing a common Lyapunov certificate for each of the objectives in the underlying analysis LMIs. One can then proceed similarly as stated in Remark \ref{SHI::rema::init} in order to generate suitable initial gains $F_1$ and $F_2$.
	\end{remark}

	\subsection{Examples}

	In order to demonstrate the dual iteration as described in this section, we consider again several examples from COMPl\textsubscript{e}ib \cite{Lei04} and compare the iteration to the following two alternative static output-feedback design approaches that can deal with multiple objectives:
	\begin{itemize}
		\item The \texttt{systune} algorithm which is a derivation from \texttt{hinfstruct} from \cite{ApkNol06, Apk13} available in Matlab, using default options.
		\item \texttt{hifoo 3.5} with \texttt{hanso 2.01} from \cite{BurHen06, GumHen09} using default options.
	\end{itemize}
	As our second objective we consider here two scenarios. Both describe mandatory energy gain constraints with the matrix $P$ in \eqref{SMO::eq::scnd_constr} being chosen as
	\begin{equation*}
		P_1 := \mat{cc}{I_1 & 0 \\ 0 & -2^2 I_1}
		\teq{ and }
		P_2 := \mat{cc}{I_2 & 0 \\ 0 & -4^2 I_2}, 
	\end{equation*}
	respectively. The system matrices provided by COMPl\textsubscript{e}ib are partitioned accordingly in order to fit to the description \eqref{SMO::eq::sys_of}. The upper bounds on $\ga_\opt$ resulting from \texttt{systune} and \texttt{hifoo} are denoted as $\ga_{\mathrm{stu}}$ and $\ga_{\mathrm{hfo}}$, respectively. Moreover, let us denote by $\ga_{\mathrm{lb}}$ the lower bound on $\ga_\opt$ that is obtained by performing a standard $H_\infty$-design involving a dynamic controller for the system \eqref{SMO::eq::sys_of} without the channel from $d_2$ to $e_2$. Finally, we also determine the optimal gain bounds resulting from dynamic mixed controller synthesis\cite{SchGah97} and denote these by $\ga_{\mathrm{dm}}$; note that $\ga_{\mathrm{lb}} \leq \ga_{\mathrm{dm}}$ holds, but $\ga_{\mathrm{dm}} \leq \ga_\opt$ is not true in general due to the choice of a common Lyapunov certificate in the mixed design.
	
	
	The related numerical results are depicted in Table~\ref{SMO::tab::results_sof} and show that the upper bounds achieved by the dual iteration are close to the ones obtained by \texttt{systune} for most of the examples. \texttt{hifoo} does not seem to perform well for some of the considered examples and often results in more conservative upper bounds. 
	We also observe that the dual iteration tends to require more iterations until convergence if compared to single objective design problems. However, this is also true for the other two algorithms and due to the more difficult synthesis problem. Finally, note that, as already mentioned in the case of a single objective, all of the algorithms can profit from allowing more iterations or applying (randomized) restarting techniques at the expense of additional computation time.

	The numbers $T_{\ga^9}$ $T_{\ga^{41}}$, $T_{\ga_{\mathrm{stu}}}$ and $T_{\ga_{\mathrm{hfo}}}$ in Table~\ref{SMO::tab::results_sof} denote the average runtime for twenty runs in seconds required to compute $\ga^9$, $\ga^{41}$, $\ga_{\mathrm{stu}}$ and $\ga_{\mathrm{hfo}}$, respectively.
	For the considered examples, which all admit a relative small McMillan degree $n$, we observe that our implementation of the dual iteration as given in Algorithm \ref{SMO::algo::dual_iteration} is mostly slower than \texttt{systune} and mostly faster than \texttt{hifoo}. Recall that there are possibilities to reduce the computational burden for the dual iteration, but these are not discussed in this paper.
	As in the previous section, the running time of the algorithms \texttt{systune} and \texttt{hifoo} scales more nicely with the number of states $n$ of \eqref{SMO::eq::sys_of} since both algorithms are rather specialized and not based on solving LMIs.
	We emphasize that the flipside of this specialization is that these algorithms are (much) less amenable for various practical relevant generalization if compared to 
	the dual iteration.
	
	
	\begin{table}
		\newcommand{\f}[1]{\bf{#1}}
		\newcommand{\gr}[1]{\textcolor{gray}{#1}}
		\caption{Optimal closed-loop $H_\infty$-norms resulting from dynamic output-feedback design, upper bounds obtained via the dual iteration, \texttt{systune} and \texttt{hifoo} for several examples from \cite{Lei04} as well as the corresponding average running times within twenty runs in seconds. All values are rounded to two decimals.}%
		\label{SMO::tab::results_sof}%
		\begin{center}%
			\setlength{\tabcolsep}{5pt}%
			\renewcommand{\arraystretch}{1.2}%
			\begin{tabular}{@{}l@{\hskip 3ex}l@{\hskip 3ex}r@{\hskip 3ex}r@{\hskip 3ex}rrrrr@{\hskip 3.5ex}rr@{\hskip 3.5ex}rr@{}}
				\toprule
				& && &\multicolumn{5}{@{}c@{\hskip 3.5ex}}{Dual Iteration} &  \multicolumn{2}{@{}c@{\hskip 3.5ex}}{\texttt{systune}} & \multicolumn{2}{@{}c@{\hskip -1ex}}{\texttt{hifoo} }\\ \cmidrule(r{3.5ex}){5-9}\cmidrule(r{3.5ex}){10-11} \cmidrule{12-13}
				Scenario & Name & $\ga_\mathrm{lb}$ & $\ga_\mathrm{dm}$ & $\ga^5$ & $\ga^9$ & $T_{\ga^9}$ & $\ga^{41}$ & $T_{\ga^{41}}$ & $\ga_\mathrm{stu}$ & $T_{\ga_{\mathrm{stu}}}$ &$\ga_{\mathrm{hfo}}$ & $T_{\ga_{\mathrm{hfo}}}$\\ \hline 
				\multirow{7}{*}{$P_1$} 
				& AC3 & 2.33 & 2.62 & 3.01 & 2.96 & \gr{1.01} & 2.93 & \gr{4.16} & \f{2.44} & \gr{0.69} & 16.02 &  \gr{14.92} \\
				& AC8 & 0.51 & 1.31 & 0.70 & 0.64 & \gr{3.78} & \f{0.51} & \gr{13.86} & \f{0.51} & \gr{0.15} & 1.07 & \gr{11.35} \\
				& REA1 & 0.15 & 0.27 & 0.19 & \f{0.18} & \gr{0.41} & \f{0.18} & \gr{1.53}& \f{0.18} & \gr{0.18} & 0.22 & \gr{9.40} \\
				& WEC3 & 3.38 & 3.38 & 4.60 & 4.08 & \gr{13.01} & \f{3.69} & \gr{39.42} & \f{3.69} & \gr{0.77} & 3.94 & \gr{60.46}\\
				& IH & 0.00 & 0.61 & 0.56 & 0.03 & \gr{975.49} & \f{0.00} & \gr{1460.60} & \f{0.00} & \gr{17.45} & 112.59 & \gr{40.99} \\
				& EB1 & 0.00 & 2.36 & \f{2.21} & \f{2.21} & \gr{3.10}& \f{2.21} & \gr{5.32}& \f{2.21} & \gr{0.53}& \f{2.21} & \gr{1.93} \\
				& NN16 & 0.16 & 0.71 & 0.90 & 0.85 & \gr{2.78} & 0.76 & \gr{7.04} & \f{0.54} & \gr{1.29} & 0.90 & \gr{50.30} \\ \hdashline
				\multirow{8}{*}{$P_2$} 
				& AC3 & 1.01 & 1.09 & 1.14 & 1.10 & \gr{0.82} & \f{1.04} & \gr{3.02} & 1.08 & \gr{0.57} & 5.69 &  \gr{8.04} \\
				& AC6 & 1.01 & 2.04 & 2.48 & 2.02 & \gr{2.28} & 1.27 & \gr{7.81} & \f{1.12} & \gr{0.21}  & 1.13 &  \gr{29.06} \\
				& AC17 & 0.68 & 3.51 & 1.87 & 1.87 & \gr{0.34} & 1.85 & \gr{1.27} & \f{1.44} & \gr{0.10} & \f{1.44} & \gr{0.59} \\
				& HE2 & 0.63 & 1.42 & 2.40 & 2.40 & \gr{0.39} & 2.38 & \gr{1.13} & \f{2.36} & \gr{0.16}& 2.90 & \gr{22.58} \\
				& REA2 & 0.00 & 0.23 & 0.18 & 0.18 & \gr{0.38} & 0.15 & \gr{1.27} & \f{0.11} & \gr{0.73} & 0.18 & \gr{8.61} \\
				& DIS3 & 0.71 & 0.91 & 1.04 & 1.01 & \gr{1.30} & 0.96 & \gr{4.73} & \f{0.76} & \gr{1.76} & 1.04 & \gr{13.33} \\
				& WEC1 & 2.33 & 2.33 & 4.39 & 3.69 & \gr{13.07} & \f{3.25} & \gr{39.55}& \f{3.25} & \gr{0.64} & 3.39 & \gr{23.91} \\
				& NN4 & 1.01 & 1.01 & 1.57 & 1.49 & \gr{0.42} & 1.20 & \gr{1.58} & \f{1.00} & \gr{0.47}& 1.02 & \gr{20.27} \\
				\bottomrule
			\end{tabular}
		\end{center}
	\end{table}

	\section{Robust Output-Feedback $H_\infty$-Design}\label{RS::sec::rs}
	
	This section deals with robust dynamic output-feedback controller synthesis, which is closely related to static output-feedback design as discussed earlier in terms of
	reasons for non-convexity. Due to the importance of the underlying design problem, we provide the details for the corresponding dual iteration. Thereby, we focus on a performance criterion in terms of the energy gain as in Section~\ref{SHI::sec::shi}, since this enables to use the elimination lemma \ref{RS::lem::elimination} throughout. In particular, we demonstrate that both the static and the robust design are dealt within a common synthesis framework based on linear fractional representations. We briefly demonstrate later on that this framework even encompasses the design of robust gain-scheduling controllers.
	
	\subsection{Analysis}\label{RS::sec::ana}
	
	For some real matrices of appropriate dimensions and an initial condition $x(0) \in \R^n$, we consider the feedback interconnection 
	\begin{equation}
		\arraycolsep=3pt
		\mat{c}{\dot x(t) \\ \hline z(t) \\ e(t)} = \mat{c|cc}{A & B_1 & B_2 \\ \hline C_1 & D_{11} & D_{12} \\ C_2 & D_{21} & D_{22}} \mat{c}{x(t) \\ \hline w(t) \\ d(t)}, \quad
		w(t) = \Del(t) z(t), 
		\label{RS::eq::sys}
	\end{equation}
	for $t\geq 0$; here, $d\in L_2$ is a generalized disturbance, $e$ is the performance output desired to be small, $w, z$ are interconnection variables and $\Del$ is a time-varying uncertainty contained in the set
	\begin{equation*}
		\Delf(\Vb) := \{ \Del: [0, \infty) \to  \Vb~|~ \Del \text{ is piecewise continuous} \}
	\end{equation*}
	for some known compact value set $\Vb \subset \R^{q\times p}$. In particular, we do not make any assumptions on the rate of variation of the uncertainty $\Del$. The description \eqref{RS::eq::sys} is called linear fractional representation (LFR) as closing the loop involving the signals $z$ and $w$ leads to a linear parameter-varying system 
	where $\Del$ enters in a rational fashion \cite{ZhoDoy96, SchWei00, Hof16}.
	%
	
	\begin{definition}
		\label{RS::def::stab}
		%
		The system \eqref{RS::eq::sys} is said to be robustly stable if $\det(I -  D_{11}\Del) \neq 0$ for all $\Del \in \Vb$ and if there exist constants $M, \la > 0$ such that
		\begin{equation*}
			\|x(t)\| \leq Me^{-\la t} \|x(0)\| \text{ ~~~for all } t \geq 0, 
			\text{ ~all }\Del\in \Delf(\Vb),
			\text{ ~all initial conditions }x(0) \in \R^n
			\text{ ~and for }d = 0.
		\end{equation*}
		%
		It is said to admit a robust energy gain smaller than $\ga>0$ if it is robustly stable and there exists an $\eps > 0$ such that
		\begin{equation*}
			\|e\|_{L_2}^2 \leq (\ga^2 - \eps) \|d\|_{L_2}^2
			\text{ ~for all }d \in L_2,
			\text{ ~all } \Del \in \Delf(\Vb)
			\text{ ~and for }x(0) = 0.
		\end{equation*}
		The infimum of all such values $\ga > 0$ is the system's robust energy gain.
	\end{definition}
	\vspace{1ex}
	
	As we are facing arbitrarily time-varying uncertainties, we employ the following consequence of the analysis result from \cite{Sch01} that relies on the full block S-procedure. It can also be viewed as a special case of the IQC result in \cite{MegRan97} with a static multiplier.
	
	\begin{lemma}
		\label{RS::lem::stab}
		Let $P_\ga := \smat{I & 0 \\ 0 & -\ga^2 I}$ and let $G_{ij} := [A, B_j, C_i, D_{ij}]$ be the transfer matrices corresponding to \eqref{RS::eq::sys}.
		Further, let $\Pb(\Vb)$ be a set of symmetric nonsingular matrices with LMI representation\footnote{This means that there exist affine matrix-valued functions $\Psi$ and $\Phi$ such that $\Pb(\Vb) = \{\Psi(\nu)~|~ \nu \in \R^\bullet \text{ ~~and~~ } \Phi(\nu) \cg 0 \}$.} such that any $P \in \Pb(\Vb)$ satisfies
		\begin{equation}
			\mat{c}{0 \\ I}^T P \mat{c}{0 \\ I} \cle 0
			\teq{ and }\mat{c}{I \\ \Del}^T P \mat{c}{I \\ \Del} \cge 0
			\text{ ~~ for all~~ }\Del \in \Vb.
			%
			\label{RS::eq::multiplier_set}
		\end{equation}
		Then the system \eqref{RS::eq::sys} admits a robust energy gain smaller than $\ga > 0$ if there exist matrices $X$ and $P \in \Pb(\Vb)$ which satisfy
		\begin{equation}
			\label{RS::lem::lmi_stab}
			X \cg 0 \teq{ and }
			\Ls\left(\mat{cc}{0 & X \\ X & 0}, \mat{c|c}{P & 0 \\ \hline 0 & P_\ga}, \mat{cc}{G_{11} & G_{12} \\ I & 0 \\ \hline G_{21} & G_{22} \\ 0 & I}_\ss \right) \cl 0.
		\end{equation}
	\end{lemma}

	Here, the matrix $P$ is usually referred to as a multiplier and $\Pb(\Vb)$,  accordingly, as the set of multipliers. The latter set should always be chosen as large as possible and hence describe $\Vb$ as good as possible in terms of quadratic inequalities. 
	As an example, suppose that $\Vb$ equals $\mathrm{co}\{\Del_1, \dots, \Del_N\}$, the convex hull of some given generators $\Del_1, \dots, \Del_N$. Then
	\begin{equation}
		\label{RS::eq::multiplier_set_for_ch}
		\Pb(\Vb) := \left\{P=P^T ~\middle| ~\mat{c}{0 \\ I}^TP \mat{c}{0\\ I}\cl 0 \text{ ~~and~~ }\mat{c}{I \\ \Del_i}^TP\mat{c}{I \\ \Del_i}\cg 0 \text{ ~~for all~~ }i = 1,\dots, N \right\}
	\end{equation}
	is a set of multipliers which, indeed, satisfies \eqref{RS::eq::multiplier_set}
	and has an LMI representation. Moreover, note that any $P \in \Pb(\Vb)$ is nonsingular as a consequence of the minimax theorem of Courant and Fischer\cite{HorJoh90}. 
	%
	%
	As another example let us suppose that $\Vb := \{ vI~:~ v \in [a, b] \}$ for some $a < b$. Then it is possible to employ the above set of multipliers for $\Vb$ as well or the commonly used alternative
	\begin{equation}
		\label{RS::eq::multiplier_set_for_int}
		\Pb(\Vb) := \left\{  \mat{cc}{bI & -I \\ -aI & I}^T \mat{cc}{0 & H^T \\ H & 0}\mat{cc}{bI & -I \\ -aI & I}~\middle| ~ H + H^T \cg 0 \right\},
	\end{equation}
	which is closely related to the set of so-called D-G scalings. 
	Several additional examples can be found, e.g., in \cite{SchWei00} and a detailed summary in the context of IQCs is available in \cite{Vee15}. In particular, note that the LMIs \eqref{RS::lem::lmi_stab} also imply a bounded robust energy gain 
	in the case that $\Del$ in \eqref{RS::eq::sys} is a time-varying static nonlinear uncertainty, if the symmetric matrix $P$ satisfies
	\begin{equation*}
		\mat{c}{z \\ \Del(t, z)}^T P \mat{c}{z \\ \Del(t, z)} \cge 0
		\teq{ for all }t\geq 0
		\teq{ and all } z \in \R^p.
	\end{equation*}

	Finally, note that a multiplier set is not required to satisfy the first LMI in \eqref{RS::eq::multiplier_set} for concluding a bounded robust energy gain via Lemma \ref{RS::lem::stab}. This constraint, which holds for many common choices such as the two given above, is only included to simplify the exposition. Otherwise, we can assure by other means that the employed multipliers have the correct amount of negative and positive eigenvalues in several of the next results.
	


	In the sequel we will also need the dual multiplier set corresponding to a given set $\Pb(\Vb)$. The latter is defined as
	\begin{equation*}
		\t \Pb(\Vb) := \{\t P~|~ \t P^{-1} \in \Pb(\Vb) \}
		\teq{ if it has an LMI representation.}
	\end{equation*}
	Note that the set $\{\t P~|~ \t P^{-1} \in \Pb(\Vb) \}$ does not have an LMI representation for any set $\Pb(\Vb)$, but in most practical situations it does. 
	For the two previous examples \eqref{RS::eq::multiplier_set_for_ch} and \eqref{RS::eq::multiplier_set_for_int}, the dual sets are explicitly given as
	\begin{equation*}
		\t \Pb(\Vb) := \left\{\t P=\t P^T ~\middle| ~\mat{c}{I \\ 0}^T\t P \mat{c}{I\\ 0}\cg 0 \text{ ~~and~~ }\mat{c}{-\Del_i^T \\ I}^T\t P\mat{c}{-\Del_i^T \\ I}\cl 0 \text{ ~~for all~~ }i = 1,\dots, N \right\}
	\end{equation*}
	and
	\begin{equation*}
		\t \Pb(\Vb) := \left\{  \frac{1}{(b-a)^2}\mat{cc}{I & I \\ aI & bI} \mat{cc}{0 & H \\ H^T & 0}\mat{cc}{I & I \\ aI & bI}^T~\middle| ~ H + H^T \cg 0 \right\},
	\end{equation*}
	respectively. 

	\subsection{Synthesis}\label{RS::sec::synth}
	
	\subsubsection{Problem Description}

	For fixed real matrices of appropriate dimensions and some initial condition $x(0) \in \R^n$, we consider the feedback interconnection
	\begin{equation}
		\arraycolsep=1pt
		\mat{c}{\dot x(t) \\\hline  z(t)\\ e(t) \\ y(t)} 
		= \mat{c|ccc}{A & B_1 & B_2 & B_3 \\ \hline C_1 & D_{11} & D_{12} & D_{13} \\ C_2 & D_{21} & D_{22} & D_{23} \\
			C_3 & D_{31} & D_{32} & 0} 
		\mat{c}{x(t) \\\hline w(t) \\ d(t) \\ u(t)}, \quad
		w(t) = \Del(t) z(t)
		\label{RS::eq::sys_of}
	\end{equation}
	for $t\geq 0$; here, $u$ is the control input, $y$ is the measured output, $\Del\in \Delf(\Vb)$ is some uncertainty and $\Vb \subset \R^{q\times p}$ is a compact value set. Further, suppose that we are given a multiplier set $\Pb(\Vb)$ corresponding to $\Vb$ as well as its dual multiplier set $\t \Pb(\Vb)$ as defined in the previous subsection.
	%
	%
	%
	Our main goal is the design of a robust dynamic output-feedback controller with the description
	\begin{equation}
		\mat{c}{\dot x_c(t) \\ u(t)} 
		= \mat{cc}{A^c &  B^c \\  C^c & D^c } \mat{c}{x_c(t) \\ y(t)}
		\label{RS::eq::con_of2}
	\end{equation}
	for the system \eqref{RS::eq::sys_of} such that the corresponding closed-loop robust energy gain is as small as possible. The latter closed-loop interconnection is described by
	\begin{equation}
		\arraycolsep=2pt
		\mat{c}{\dot x_\mathrm{cl}(t) \\\hline  z(t) \\ e(t)} 
		= \mat{c|cc}{\Ac & \Bc_1 & \Bc_2 \\ \hline 
			\Cc_1 & \Dc_{11} & \Dc_{12} \\ 
			\Cc_2 & \Dc_{21} & \Dc_{22}} 
		\mat{c}{x_\mathrm{cl}(t) \\\hline w(t) \\ d(t)}, \quad
		w(t) = \Del(t) z(t)
		\label{RS::eq::cl_of}
	\end{equation}
	with $t\geq 0$ as well as $x_\mathrm{cl} = \smat{x\\ x_c}$ and standard calligraphic closed-loop matrices. From the analysis criteria in Lemma \ref{RS::lem::stab} and by applying the elimination lemma \ref{RS::lem::elimination}, we immediately obtain the following synthesis result. 
	
	\begin{theorem}
		\label{RS::theo::of}
		Let $G_{ij}$ be the transfer matrices corresponding to \eqref{RS::eq::sys_of}. Further, let $V$ and $U$ be a basis matrices of $\ker(C_3, D_{31}, D_{32})$ and $\ker(B_3^T, D_{13}^T, D_{23}^T)$, respectively. Then there exists a controller \eqref{RS::eq::con_of2} for the system \eqref{RS::eq::sys_of} such that the analysis LMIs \eqref{RS::lem::lmi_stab} are feasible for \eqref{RS::eq::cl_of} if and only if there exist symmetric matrices $X$, $Y$ and $P\in \Pb(\Vb)$ satisfying 
		\begin{subequations}
			\label{RS::theo::eq::lmi_of}
			\begin{equation}
				\arraycolsep=1pt
				\mat{cc}{X & I \\ I & Y} \cg 0, ~~
				V^T\Ls\left(\mat{cc}{0 & X \\ X & 0}, \mat{c|c}{P & 0 \\ \hline 0 & P_\ga}, \mat{cc}{G_{11} & G_{12} \\ I & 0\\ \hline
					G_{21} & G_{22} \\ 0 & I
				}_{\ss} \right) V \cl 0 
				\text{ ~~and~~ }
				U^T\Ls\left(\mat{cc}{0 & Y \\ Y & 0}, \mat{c|c}{P^{-1} & 0 \\ \hline 0 & P_\ga^{-1}}, \mat{cc}{I & 0 \\ -G_{11}^\ast & -G_{21}^\ast \\ \hline
					0 & I \\
					-G_{12}^\ast & -G_{22}^\ast
				}_{\ss} \right) U \cg 0.
				\tlabel{RS::theo::eq::lmi_ofa}{RS::theo::eq::lmi_ofb}{RS::theo::eq::lmi_ofc}
			\end{equation}
		\end{subequations}
		In particular, the infimal $\ga > 0$ such that there exist symmetric $X, Y$ and $P \in \Pb(\Vb)$ satisfying the above inequalities is equal to
		\begin{equation*}
			\ga_\opt := \inf\left\{ \ga > 0~\middle| 
			\text{ There is a controller \eqref{RS::eq::con_of2} s.th. the analysis
				LMIs \eqref{RS::lem::lmi_stab} are feasible for \eqref{RS::eq::cl_of}}\right\}. 
			%
		\end{equation*}
	\end{theorem}
	
	Similarly as in the previous section, note that $\ga_\opt$ is \emph{not} the optimal robust energy gain achievable by robust controllers with description \eqref{RS::eq::con_of2}. This is due to the conservatism in the employed analysis result Lemma \ref{RS::lem::stab}. 
	
	In contrast to static output-feedback design as considered in Section \ref{SHI::sec::shi}, non-convexity emerges through the multiplier $P$ and its inverse appearing in \eqref{RS::theo::eq::lmi_ofb} and \eqref{RS::theo::eq::lmi_ofc} instead of the Lyapunov certificate $X$ and its inverse. Due to this non-convexity, computing $\ga_\opt$ or a corresponding controller is difficult in general. Subsequently, we modify the dual iteration in order to compute upper bounds on $\ga_\opt$ and, in particular, solve the robust output-feedback $H_\infty$-design problem
	still with performance guarantees.

	\subsubsection{Dual Iteration: Initialization}
	
	In order to initialize the dual iteration, we aim again to compute a meaningful lower bound on $\ga_\opt$. 
	Now the lower bound is obtained by considering the design of
	a gain-scheduling controller\cite{Pac94}. Such a controller is taking online measurements 
	of the uncertainty $\Del(t)$ into account and given by
	\begin{equation}
		\mat{c}{\dot x_c(t) \\ \hline z_c(t) \\ u(t)} 
		= \mat{c|cc}{\h A^c & \h B^c_1 & \h B^c_2 \\ \hline  
			\h C^c_1 & \h D^c_{11} & \h D^c_{12} \\ 
			\h C^c_2 & \h D^c_{21} & \h D^c_{22}} \mat{c}{x_c(t) \\\hline w_c(t) \\ y(t)},\quad
		w_c(t) = S(\Del(t)) z_c(t)
		\label{RS::eq::con_of3}
	\end{equation}
	for $t \geq 0$ and with some function $S$. 
	If there exists a robust controller for the system \eqref{RS::eq::sys_of} achieving a robust energy gain of $\ga$, there also exists a gain-scheduling controller
	which achieves (at least) the same robust energy gain.
	This just follows from the observation that 
	any robust controller can be expressed as \eqref{RS::eq::con_of3} with 
	\begin{equation*}
		\mat{c|c}{\h A^c & \h B^c_2 \\ \hline  
			\h C^c_2 & \h D^c_{22}}
		= \mat{c|c}{A^c & B^c  \\ \hline  C^c & D^c}, \quad
		\mat{c}{\h B^c_1 \\ \h D^c_{21}} = 0, \quad
		\mat{cccc}{\h C^c_1 & \h D^c_{11} & \h D^c_{12}} = 0
		\teq{ and }
		S(\Del) = 0
		\teq{ for all } \Del \in \Vb.
	\end{equation*}
	%
	%
	In our setup, the problem of finding such a gain-scheduling controller \eqref{RS::eq::con_of3} for the system \eqref{RS::eq::sys_of} can be turned into a convex optimization problem; the design of structured gain-scheduling controllers, e.g., with $(\h D^c_{11}, \h D^c_{12}) = 0$ and $\h D^c_{21} = 0$ would yield even better lower bounds but, unfortunately, seems to be a non-convex problem without additional structural properties of the underlying system \eqref{RS::eq::sys_of}. For unstructured gain-scheduling controller design we have the following result which is essentially taken from \cite{Sch00, Sch01}.

	\begin{theorem}
		\label{RS::theo::gs}
		Let $G_{ij}$, $U$ and $V$ be as in Theorem \ref{RS::theo::of}. Then there exists a gain-scheduling controller \eqref{RS::eq::con_of3} and a scheduling function $S$ for the system \eqref{RS::eq::sys_of} such that the analysis LMIs \eqref{RS::lem::lmi_stab} are feasible for the resulting corresponding closed-loop system, for the value set $\Vb_e = \{\diag(\Del, S(\Del))~|~ \Del \in \Vb\}$ and for a corresponding multiplier set $\Pb_e(\Vb_e)$ if and only if there exist symmetric matrices $X, Y$ and $P\in \Pb(\Vb)$, $\t P\in \t \Pb(\Vb)$ satisfying the synthesis LMIs \eqref{RS::theo::eq::lmi_of} with $P^{-1}$ in \eqref{RS::theo::eq::lmi_ofc} replaced by $\t P$.
		In particular, we have
		$\gaoptgs \leq \ga_\opt$
		for $\gaoptgs$ being the infimal $\ga > 0$ such that the latter LMIs are feasible.
	\end{theorem}
	
	We do not specify the multiplier set $\Pb_e(\Vb_e)$ since it is not relevant for our purposes and because we are only interested in the lower bound $\ga_\mathrm{gs}$; it can be constructed similarly as in \cite{Sch00}.
	%
	%
	Again, such lower bounds can be good indicators for measuring the conservatism of the upper bounds that are generated by our algorithms. 
	
	%
	%
	\vspace{1ex}

	As in both previous sections, the dual iteration is initialized by the design of a suitable full-information controller. For robust synthesis, such a controller is of the form $u = F\t y = (F_1, F_2, F_3)\t y$ with $\t y := (x^T, w^T, d^T)^T$. Hence, these controllers are even able to measure the uncertain signal $w = \Del z$ in addition to the state $x$ and the disturbance $d$.
	Synthesizing such controllers is not difficult. Indeed, an application of the elimination lemma \ref{RS::lem::elimination} leads to the following result.
	
	\begin{lemma}
		\label{RS::lem::full_info}
		There exists some full-information gain $F$ such that the analysis LMIs \eqref{RS::lem::lmi_stab} are feasible for the system 
		\begin{equation}
			\mat{c}{\dot x(t) \\ \hline z(t) \\ e(t)}
			= \mat{c|cc}{A + B_3 F_1 & B_1 + B_3F_2 & B_2 + B_3F_3 \\ \hline 
				C_1 + D_{13}F_1 & D_{11} + D_{13}F_2 & D_{12} + D_{13}F_3 \\
				C_2 + D_{23}F_1 & D_{21} + D_{23}F_2 & D_{22} + D_{23}F_3}
			\mat{c}{\dot x(t) \\ \hline w(t) \\ d(t)}
			= \left(\mat{c|cc}{A & B_1 & B_2 \\ \hline 
				C_1 & D_{11}  & D_{12} \\
				C_2  & D_{21}  & D_{22} }
			+ \mat{c}{B_3 \\ \hline D_{13} \\ D_{23}}F
			\right)
			\mat{c}{\dot x(t) \\ \hline w(t) \\ d(t)}
			\label{RS::eq::clF}
		\end{equation}
		if and only if there exist symmetric $\t P \in \t \Pb(\Vb)$ and $Y \cg 0$ satisfying \eqref{RS::theo::eq::lmi_ofc} with $P^{-1}$ replaced by $\t P$.
	\end{lemma}
	

	\subsubsection{Dual Iteration}
	
	Suppose that we have synthesized a full-information gain $F$ by Lemma \ref{RS::lem::full_info}. Then the primal synthesis LMIs corresponding to the gain $F$ and to the analysis LMIs \eqref{RS::lem::lmi_stab} are obtained in a straightforward fashion.
	
	\begin{theorem}
		\label{RS::theo::ofF}
		Let $G_{ij}$ and $V$ be as in Theorem \ref{RS::theo::of} and let $G_{ij}^F$ denote the transfer matrices corresponding to \eqref{RS::eq::clF}. Then there exists a controller \eqref{RS::eq::con_of2} for the system \eqref{RS::eq::sys_of} such that the analysis LMIs \eqref{RS::lem::lmi_stab} are feasible for the corresponding closed-loop system if there exist symmetric matrices $X$, $Y$ and $P\in \Pb(\Vb)$ satisfying
		\begin{subequations}
			\label{RS::theo::eq::lmi_ofF}
			\begin{equation}
				\arraycolsep=3pt
				\mat{cc}{X & Y \\ Y & Y} \cg 0,\quad
				V^T\Ls\left(\mat{cc}{0 & X \\ X & 0}, \mat{c|c}{P & 0 \\ \hline 0 & P_\ga}, \mat{cc}{G_{11} & G_{12} \\ I & 0 \\ \hline
					G_{21} & G_{22} \\ 0 & I
				}_{\ss} \right) V \cl 0
				\teq{ and }
				\Ls\left(\mat{cc}{0 & Y \\ Y & 0}, \mat{c|c}{P & 0 \\ \hline 0 & P_\ga}, \mat{cc}{G_{11}^F & G_{12}^F \\ I & 0\\ \hline
					G_{21}^F & G_{22}^F \\
					0 & I
				}_{\ss} \right) \cl 0.
				\tlabel{RS::theo::eq::lmi_ofFa}{RS::theo::eq::lmi_ofFb}{RS::theo::eq::lmi_ofFc}
			\end{equation}
		\end{subequations}
		Moreover, we have 
		$\gaoptgs \leq \ga_\opt \leq \ga_F$
		for $\ga_F$ being the infimal $\ga > 0$ such that the above LMIs are feasible.
	\end{theorem}
	
	\begin{proof}
		Since we have $P \in \Pb(\Vb)$, we can conclude that $P$ has exactly $p$ positive and $q$ negative eigenvalues. This allows us to eliminate the full-information gain $F$ from the LMI \eqref{RS::theo::eq::lmi_ofFc} which leads to \eqref{RS::theo::eq::lmi_ofc} for $Y$ replaced by $Y^{-1}$. Finally, performing a congruence transformation of \eqref{RS::theo::eq::lmi_ofFa} with $\diag(I, Y^{-1})$ yields \eqref{RS::theo::eq::lmi_ofc} for $Y$ replaced by $Y^{-1}$. Since we have \eqref{RS::theo::eq::lmi_ofb} and $P \in \Pb(\Vb)$ by assumption, we can apply Theorem \ref{RS::theo::of} in order to construct the desired controller \eqref{RS::eq::con_of2}.
	\end{proof}
	
	The employed dual versions of Lemma \ref{RS::lem::full_info} and Theorem \ref{RS::theo::ofF} are given next.
	
	\begin{lemma}
		\label{RS::lem::full_actu}
		There exists some full-actuation gain $E$ such that the analysis LMIs \eqref{RS::lem::lmi_stab} are feasible for the system 
		\begin{equation}
			\mat{c}{\dot x(t) \\ \hline z(t) \\ e(t)}
			= \mat{c|cc}{A + E_1 C_3 & B_1 + E_1D_{31} & B_2 + E_1 D_{32} \\ \hline 
				C_1 + E_2C_3 & D_{11} + E_2D_{31} & D_{12} + E_2 D_{32} \\
				C_2 + E_3C_3 & D_{21} + E_3 D_{31} & D_{22} + E_3D_{32}}
			\mat{c}{\dot x(t) \\ \hline w(t) \\ d(t)}
			=\left(\mat{c|cc}{A  & B_1  & B_2  \\ \hline 
				C_1  & D_{11}  & D_{12}  \\
				C_2  & D_{21}  & D_{22} }
			+ E \mat{c|cc}{C_3 & D_{31} & D_{32}}
			\right)
			\mat{c}{\dot x(t) \\ \hline w(t) \\ d(t)}
			\label{RS::eq::clE}
		\end{equation}
		if and only if there exist symmetric $P\in \Pb(\Vb)$ and $X \cg 0$ satisfying \eqref{RS::theo::eq::lmi_ofb}.
	\end{lemma}
	
	\begin{theorem}
		\label{RS::theo::ofE}
		Let $G_{ij}$ and $U$ be as in Theorem \ref{RS::theo::of} and let $G_{ij}^E$ denote the transfer matrices corresponding to \eqref{RS::eq::clE}. Then there exists a controller \eqref{RS::eq::con_of2} for the system \eqref{RS::eq::sys_of} such that the analysis LMIs \eqref{RS::lem::lmi_stab} are feasible for the corresponding closed-loop system if there exist symmetric matrices $X$, $Y$ and $\t P \in \t \Pb(\Vb)$ satisfying 
		\begin{subequations}
			\label{RS::theo::eq::lmi_ofE}
			\begin{equation}
				\arraycolsep=1pt
				\mat{cc}{X & X \\ X & Y} \cg 0, ~~
				\Ls\left(\mat{cc}{0 & X \\ X & 0}, \mat{c|c}{\t P & 0 \\ \hline 0 & P_\ga^{-1}}, \mat{cc}{I & 0 \\ -(G_{11}^E)^\ast & -(G_{21}^E)^\ast \\ \hline
					0 & I \\
					-(G_{12}^E)^\ast & -(G_{22}^E)^\ast
				}_{\ss} \right) \cg 0
				\text{ ~~and~~ }
				U^T\Ls\left(\mat{cc}{0 & Y \\ Y & 0}, \mat{c|c}{\t P & 0 \\ \hline 0 & P_\ga^{-1}}, \mat{cc}{I & 0 \\ -G_{11}^\ast & -G_{21}^\ast \\ \hline
					0 & I \\
					-G_{12}^\ast & -G_{22}^\ast
				}_{\ss} \right) U \cg 0.
				\tlabel{RS::theo::eq::lmi_ofEa}{RS::theo::eq::lmi_ofEb}{RS::theo::eq::lmi_ofEc}
			\end{equation}
		\end{subequations}
		Moreover, we have $\gaoptgs \leq \ga_\opt \leq \ga_E$ for $\ga_E$ being the infimal $\ga > 0$ such that the above LMIs are feasible.
	\end{theorem}

	Theorems \ref{RS::theo::ofF} and \ref{RS::theo::ofE} are again nicely intertwined,
	in analogy of what has been stated in Theorem \ref{SHI::theo::it_summary} in Section \ref{SHI::sec::shi}. 
	In particular, the following dual iteration generates a monotonically decreasing sequence $(\ga^k)_{k \in \N}$ of upper bounds on $\ga_\opt$. Moreover, essentially the same statements as in Remark \ref{SHI::rema::stuff} can be made. 
	
	\begin{Algorithm}
		\label{RS::algo::dual_iteration}
		Dual iteration for robust output-feedback $H_\infty$-design.
		\begin{enumerate}
			\item \emph{Initialization:} Compute the lower bound $\gaoptgs$ based on solving the gain-scheduling synthesis LMIs in Theorem \ref{RS::theo::gs} and set $\ga^0 := +\infty$ as well as $k = 1$.
			Design an initial full-information gain $F$ from Lemma \ref{RS::lem::full_info}.
			\item \emph{Primal step:} Compute $\ga_F$ based on solving the primal synthesis LMIs \eqref{RS::theo::eq::lmi_ofF} for the given gain $F$ and choose some small $\eps_k>0$ such that $\ga^k := \ga_F(1+\eps_k) < \ga^{k-1}$. For $\ga = \ga^k$, determine matrices $X,Y$ and $P \in \Pb(\Vb)$ satisfying the LMIs \eqref{RS::theo::eq::lmi_ofF} and apply Lemma \ref{RS::lem::full_actu} in order to design a gain $E$ satisfying the dual synthesis LMIs \eqref{RS::theo::eq::lmi_ofE} for $(X, Y, \t P) = (X^{-1}, Y^{-1}, P^{-1})$. 
			\item \emph{Dual step:} Compute $\ga_E$ based on solving the dual synthesis LMIs \eqref{RS::theo::eq::lmi_ofE} for the given gain $E$ and choose some small $\eps_{k+1}> 0$ such that $\ga^{k+1} := \ga_E (1+\eps_{k+1}) < \ga^k$. For $\ga = \ga^{k+1}$, determine matrices $X, Y$ and $\t P \in \t \Pb(\Vb)$ satisfying the LMIs \eqref{RS::theo::eq::lmi_ofE} and apply Lemma \ref{RS::lem::full_info} in order to design a gain $F$ satisfying the primal synthesis LMIs  \eqref{RS::theo::eq::lmi_ofF} for $(X, Y, P) = (X^{-1}, Y^{-1}, \t P^{-1})$.
			\item \emph{Termination:} If $k$ is too large or $\ga^k$ does not decrease any more, then stop and construct a robust output-feedback controller \eqref{RS::eq::con_of2} for the system \eqref{RS::eq::sys_of} according to Theorem \ref{RS::theo::ofE}. \\
			Otherwise set $k = k+2$ and go to the primal step. 
		\end{enumerate}
	\end{Algorithm}
	
	\begin{remark}
		\label{RS::rema::extensions}
		\begin{enumerate}[(a)]
			\item Algorithm \ref{RS::algo::dual_iteration} can be modified in a straightforward fashion to cope with the even more challenging design of static robust output-feedback controllers. This is essentially achieved by replacing \eqref{RS::theo::eq::lmi_ofFa} and \eqref{RS::theo::eq::lmi_ofEa} with $X = Y \cg 0$ during the iteration. For the initialization, we recommend to take the additional considerations in Remark \ref{SHI::rema::init} into account.
			\item It is not difficult to extend Algorithm \ref{RS::algo::dual_iteration}, e.g., to the more general and highly relevant design of robust gain-scheduling controllers \cite{VeeSch12, Hel95}. For this problem, the uncertainty $\Del(t)$ in the description \eqref{RS::eq::sys_of} is replaced by $\diag(\Del_u(t), \Del_s(t))$ with $\Del_u(t)$ being unknown, while $\Del_s(t)$ is measurable online and taken into account by the to-be-designed controller. As for robust design, this synthesis problem is known to be convex only in very specific situations; for example if the control channel is unaffected by uncertainties \cite{VeeSch12}.
			
			An interesting special case of the general robust gain-scheduling design is sometimes referred to as inexact scheduling \cite{SatPea18}. As for standard gain-scheduling it is assumed that a parameter dependent system \eqref{RS::eq::sys_of} is given, but that the to-be-designed controller only receives noisy online measurements of the parameter instead of exact ones.
			
			We emphasize that such modifications are all straightforward to handle, due to the flexibility of the design framework based on linear fractional representations and the employed multiplier separation techniques underlying Lemma \ref{RS::lem::stab}.
		\end{enumerate}
	\end{remark}

	\subsubsection{Dual Iteration: An Alternative Initialization}\label{RS::sec::alternative_init}
	
	It can happen that the LMIs appearing in the primal step of algorithm \ref{RS::algo::dual_iteration} are infeasible for the initially designed full-information gain. 
	In order to promote the feasibility of these LMIs, we propose an alternative initialization that relies on the following result. 
	
	\begin{lemma}
		\label{RS::lem::initialization}
		Suppose that the gain-scheduling synthesis LMIs in Theorem \ref{RS::theo::gs} are feasible, that some full-actuation gain $E$ is designed from Lemma \ref{RS::lem::full_actu}, and let $G_{ij}$, $G_{ij}^E$ as well as $U$ be taken as in Theorem \ref{RS::theo::ofE}. Then there exist some $\alpha > 0$, symmetric $X, Y$ and $P, \t P \in \t \Pb(\Vb)$ satisfying the LMIs \eqref{RS::theo::eq::lmi_ofE} with $\t P$ in  \eqref{RS::theo::eq::lmi_ofEb} replaced by $P$ and 
		\begin{equation}
			\mat{cc}{\alpha I & P - \t P \\ P - \t P & I} \cg 0.
			\label{RS::lem::eq::lmi_fab}
		\end{equation}
	\end{lemma}
	
	Note that, with a Schur complement argument, \eqref{RS::lem::eq::lmi_fab} is equivalent to $\|P - \t P\|^2 < \alpha$.
	Thus by minimizing $\alpha > 0$ subject to the above LMIs, we push the two multipliers $P$ and $\t P$ as close together as possible. Due to the continuity of the map $M \mapsto M^{-1}$, this means that the inverses $P^{-1}$ and $\t P^{-1}$ are close to each other as well. We can then design a corresponding full-information gain $F$ based on Lemma \ref{RS::lem::full_info} for which the LMIs \eqref{RS::theo::eq::lmi_ofF} are very likely to be feasible for the single multiplier $P^{-1} \approx \t P^{-1}$.
	
	\begin{remark}
		\begin{enumerate}[(a)]
			\item In the case that the above procedure does not yield a gain $F$ for which the LMIs \eqref{RS::theo::eq::lmi_ofF} are feasible, one can, e.g., iteratively double $\ga$ and retry until a suitable gain is found. This practical approach works typically well in various situations.
			\item It would be nicer to directly employ additional constraints for the gain-scheduling synthesis LMIs in Theorem \ref{RS::theo::gs} which promote $P \approx \t P^{-1}$ and, thus, the feasibility of the primal synthesis LMIs \eqref{RS::theo::eq::lmi_ofF} similarly as it was possible for static design in Remark~\ref{SHI::rema::init}. However, as far as we are aware of, this is only possible for specific multipliers and corresponding value sets.
		\end{enumerate}
	\end{remark}

	\subsection{Examples}
	
	Unfortunately, we can not provide a fair comparison of the generalized dual iteration presented in this section with the non-LMI based algorithms \texttt{systune} \cite{ApkNol06} and \texttt{hifoo} \cite{GumHen09}, because, to the best of the authors' knowledge, these do not apply for problems involving time-varying parametric or nonlinear uncertainties.
	Note that the dual iteration in Algorithm \ref{RS::algo::dual_iteration} is also applicable for systems \eqref{RS::eq::sys_of} affected by time-invariant parametric uncertainties. However, the underlying analysis result in Lemma \ref{RS::lem::stab} does not properly take time-invariance into account and is, hence, typically rather conservative for these types of uncertainties.  
	A comparison to other LMI based approaches, such as the D-K iteration, e.g., as suggested in Chapter 7 of \cite{SchWei00}, or a scheme based on S-variables \cite{EbiPea15} would be possible and fair. 
	However, for numerous modified examples from COMPl\textsubscript{e}ib \cite{Lei04}, we obtain qualitatively very similar results to the ones obtained in Section \ref{SHI::sec::exa} 
	that show a better performance of the dual iteration, which is the reason to omit such a comparison for brevity.
	As the only only remarkable difference, the dual iteration is now faster than the D-K iteration and the scheme based on S-variables. This is due to the fact that, in this section, we consider the design of dynamic controllers with the same number of states $n$ as the original system \eqref{RS::eq::sys_of}. Then the latter two algorithms still involve LMIs with Lyapunov matrices in $\R^{2n \times 2n}$ corresponding to the closed-loop interconnection, which is in contrast to the smaller dimension of the Lyapunov matrices in $\R^{n \times n}$ appearing in the Algorithm \ref{RS::algo::dual_iteration}.

	Instead of providing further numerical comparisons, we rather consider an interesting missile control problem and illustrate the extension of Algorithm \ref{RS::algo::dual_iteration} to robust gain-scheduling controller synthesis which constitutes an even more intricate problem than robust design.

	\subsubsection{Missile Control Problem}
	
	Similarly as, e.g., in \cite{ BalPac92, Hel95, SchNji97, PrePos08} and after some simplifications, we face a nonlinear state space model of the form
	\begin{equation}
		\label{RS::exa::mis}
		\begin{aligned}
			\dot \alpha(t) &= K_\alpha M(t) \left[\left(a_n |\alpha(t)|^2 + b_n |\alpha(t)| + c_n\left(2-\frac{M(t)}{3}\right) \right)\alpha(t) + d_n \delta(t) \right] + q(t) \\
			\dot q(t) &= K_q M(t)^2 \left[\left(a_m |\alpha(t)|^2 + b_m |\alpha(t)| + c_m \left(-7 + \frac{8M(t)}{3} \right) \right)\alpha(t) + d_m \delta(t) \right] \\
			n(t) &= K_n M(t)^2 \left[\left(a_n |\alpha(t)|^2 + b_n |\alpha(t)| + c_n\left(2+\frac{M(t)}{3}\right) \right)\alpha(t) + d_n \delta(t) \right],
		\end{aligned}
	\end{equation}
	where $M(t)$ is the Mach number assumed to take values in the interval $[2, 4]$ and with signals
	\begin{center}
		\begin{tabular}{ll@{\hskip 4ex}ll}
			$\alpha$ & angle of attack (in rad) &
			$q$ & pitch rate (in rad/s) \\
			$\delta$ & tail fin deflection (in rad) &
			$n$ & normal acceleration of the missile  (in $g = 32.2$ $\mathrm{ft}/ \mathrm{s}^2$). \\
		\end{tabular}
	\end{center}
	Note that \eqref{RS::exa::mis} is a reasonable approximation for $\alpha(t)$ between $-20$ and $20$ degrees, i.e., $|\alpha(t)| \in [0, \pi/9]$. The constants are given by
	\begin{center}
		\begin{tabular}{llll}
			$a_n = 0.000103 \cdot (180/\pi)^3$&
			$b_n = -0.00945 \cdot (180/\pi)^2$&
			$c_n = -0.1696 \cdot (180/\pi)$&
			$d_n = -0.034 \cdot(180/\pi)$ \\
			$a_m = 0.000215 \cdot (180/\pi)^3$ &
			$b_m = -0.0195 \cdot (180/\pi)^2$&
			$c_m = 0.051 \cdot (180/\pi)$&
			$d_m = -0.206 \cdot (180/\pi)$ \\
			$K_\alpha =  0.7  P_0 S / mv_s$&
			$K_q =  0.7 P_0 S d / I_y \teq{ and }$&
			$K_n = 0.7 P_0 S / mg$.& 
		\end{tabular}
	\end{center}
	The terms in the latter three constants are
	\begin{center}
		\begin{tabular}{ll@{\hskip6ex}ll}
			$P_0 = 973.3$ $\mathrm{lbf}/\mathrm{ft}^2$ & static pressure at 20,000 $\mathrm{ft}$ & 
			$S = 0.44$ $\mathrm{ft}^2$ & reference area \\
			$m = 13.98$ $\mathrm{slugs}$ & mass of the missile  & 
			$v_s = 1036.4$ $\mathrm{ft}/ \mathrm{s}$ & speed of sound at 20,000 $\mathrm{ft}$ \\
			$d = 0.75$ $\mathrm{ft}$ & diameter  &
			$I_y = 182.5$ $\mathrm{slug} \cdot \mathrm{ft}^2$ & pitch moment of inertia. \\
		\end{tabular}
	\end{center}
	The goal is to find a controller such that the commanded acceleration maneuvers $n_c$ are tracked and such that the physical limitations of the fin actuator are not exceeded. Precisely, the objectives are:
	\begin{itemize}
		\item rise-time less than $0.35$ $\mathrm{s}$, 
		steady state error less than $1\%$ and 
		overshoot less than $10\%$.
		\item tail fin deflection less than $25$ $\mathrm{deg}$ and
		tail fin deflection rate less than $25$ $\mathrm{deg}/ \mathrm{s}$ per commanded $g$.
	\end{itemize}
	First, we assume that $\alpha$, $n_c - n$ and $M$ are available for control and, similarly as in \cite{ BalPac92, Hel95, SchNji97, PrePos08}, design a gain-scheduling controller. The latter controller will depend in a nonlinear fashion on the parameters $\alpha$ and $M$ in \eqref{RS::exa::mis}.
	To this end we can rewrite \eqref{RS::exa::mis} as
	\begin{equation*}
		\mat{c}{\dot \alpha(t) \\ \dot q(t) \\ \hline z(t) \\ \hdashline n(t) \\ \alpha(t)}
		={\small \mat{cc|ccccccc:c}{
				0 & 1 & 0 & 0 & 0 & K_\alpha & 0 & 0 & 0 & 0 \\
				0 & 0 & 0 & 0 & 0 & 0 & 0 & 0 & K_q & 0 \\ \hline
				0 & 0 & 0 & 1 & 0 & 0 & 0 & 0 & 0 & 0 \\
				1 & 0 & 0 & 0 & 0 & 0 & 0 & 0 & 0 & 0 \\
				1 & 0 & 0 & 0 & 0 & 0 & 0 & 0 & 0 & 0 \\
				2c_n & 0 & a_n & b_n & -\tfrac{c_n}{3} & 0 & 0 & 0 & 0 & d_n \\
				0 & 0 & 0 & 0 & 0 & 1 & 0 & 0 & 0 & 0 \\
				-7c_m & a_m & b_m & \tfrac{c_m8}{3} & 0 & 0 & 0 & 0 & 0 & d_m \\
				0 & 0 & 0 & 0 & 0 & 0 & 0 & 1 & 0 & 0 \\ \hdashline
				0 & 0 & 0 & 0 & 0 & 0 & K_n & 0 & 0 & 0 \\
				1 & 0 & 0 & 0 & 0 & 0 & 0 & 0 & 0 & 0}}
		\mat{c}{\alpha(t) \\ q(t) \\ \hline w(t) \\ \del(t)},\qquad
		w(t) = \underbrace{\mat{cc}{|\alpha(t)|~ I_2 &  \\ & M(t) I_5}}_{=:\Delta(t)} z(t)
	\end{equation*}
	which includes the measurable signal $\alpha$ as an output. In particular, the above system is the feedback interconnection of an LTI plant $P$ and a time-varying operator $\Delta$. Following \cite{SchNji97}, we aim to design a controller that ensures that the closed-loop specifications are satisfied, by considering the weighted synthesis interconnection as depicted in Fig. \ref{RS::fig::missile_exa}. Here, the fin is driven by the output of $G_\mathrm{act}$, an actuator of second order modeled as
	\begin{equation*}
		\mat{c}{\dot x_\mathrm{act}(t) \\ \hline \del(t) \\ \dot \del(t)}
		= \mat{c|c}{A_\mathrm{act} & B_\mathrm{act} \\ \hline 
			C_\mathrm{act} & 0 \\ C_\mathrm{act}A_\mathrm{act} & C_\mathrm{act}B_\mathrm{act}} \mat{c}{x_\mathrm{act}(t) \\ u(t)}
		\teq{ where }
		C_\mathrm{act}(sI - A_\mathrm{act})^{-1}B_\mathrm{act}
		= \frac{(150)^2}{s^2 + 2\cdot 150\cdot 0.7s + (150)^2}.
	\end{equation*}
	The exogenous disturbances $d_1$ and $d_2$ are used to model measurement noise. 
	The ideal model and weighting filters are given by
	\begin{equation*}
		G_\mathrm{id}(s) = \frac{144(1 -0.05s)}{s^2 + 19.2s+144},\quad
		W_e(s) = \frac{0.5s + 17.321}{s + 0.0577},\quad
		W_\delta(s) = \frac{1}{19},\quad
		W_{\dot \delta}(s) = \frac{1}{25} \teq{ and }
		W_{d_1} = W_{d_2} = 0.001.
	\end{equation*}
	%
	Disconnecting the controller $\Delta \star K$ and, e.g., using the Matlab command \texttt{sysic} yields a system with description \eqref{RS::eq::sys_of} with the stacked signals $d := \col(n_c, W_{d_1}d_1, W_{d_2}d_2)$, $e := \col(W_e(n_\mathrm{id} - n), n, W_\delta \delta, W_{\dot \delta} \delta)$, $y := \col(n_c - n, \alpha)$ and the value set
	\begin{equation*}
		\Vb :=  \{\diag(\del_1 I_2, \del_2 I_5))~\mid~ \del_1 \in [0, \pi/9] \text{ ~and~ } \del_2 \in [2, 4] \}.
	\end{equation*}
	We can hence use a set of multipliers similar to the one in \eqref{RS::eq::multiplier_set_for_int} (closely related to D-G scalings) and employ our analysis and design results. For the synthesis of a gain-scheduling controller, we make use of Theorem \ref{RS::theo::gs}; note that for D-G scalings it is possible to use the scheduling function $S(\Del) = \Del = \mathrm{id}(\Del)$.
	
	Applying Theorem \ref{RS::theo::gs} yields an upper bound on the optimal closed-loop energy gain of $\ga_{\mathrm{gs}} = 2.23$ and the upper row of Fig.~\ref{RS::fig::BodeGS} depicts the Bode magnitude plots of the corresponding closed-loop system with the resulting gain-scheduling controller for frozen values of $\Delta$. Finally, time-domain simulations of the nonlinear closed-loop systems are given in the first two rows of Fig.~\ref{RS::fig::MISGS}. Here, we consider trajectories for several (almost arbitrarily chosen) Mach numbers
	\begin{equation}
		M_k(t) = \mathrm{sat}(4 - \tfrac{(t + 1.25(k-1))}{5})
		\teq{ with }
		\mathrm{sat}(t) = \max\big\{2,~ \min\{4, t\} \big\}
		\label{RS::eq::exa::Machnumbers}
	\end{equation}
	and we let both disturbances $d_1$ and $d_2$ be zero.
	We observe that the specifications are met for most of those Mach numbers apart from the constraint on the tail fin deflection rate, which is not well-captured by $H_\infty$-criteria. Of course, the performance of the designed controller can be improved by readjusting the weights, but this is not our intention at this point. 
	
	\vspace{1ex}
	
	Instead, let us now assume that the Mach number $M$ can not be measured online and that only $\alpha$ and $n_c - n$ are available for control. Hence, we now aim to design a controller that is robust against variations in $M$, but benefits from the measurement of the parameter $\alpha$ that enters \eqref{RS::exa::mis} in a nonlinear fashion.
	This boils down to the synthesis of a robust gain-scheduling controller, which is more general and challenging than the design of robust controllers as considered in this section. 
	However, as emphasized in Remark \ref{RS::rema::extensions} and due to the modularity of the LFR framework, it is fortunately not difficult to extend the dual iteration in order to cope with such a problem as well.
	Indeed, after five iterations we reach an upper bound of $\ga^5 = 3.30$ on the optimal closed-loop robust energy gain, which is not far away from the bound achieved by the gain-scheduling design.
	The lower row of Fig.~\ref{RS::fig::BodeGS} illustrates the resulting closed-loop frequency responses for several frozen values of $\Del$ and last two rows of Fig. \ref{RS::fig::MISGS} depicts simulations of the nonlinear closed-loop system for several Mach numbers as in \eqref{RS::eq::exa::Machnumbers}. 
	In particular, we observe that the tracking behavior degrades, which is not surprising as the controller takes fewer measurements into account.
	
	Finally, note that we can of course also view both $|\alpha|$ and $M$ as an uncertainty and design a robust controller based on the dual iteration as discussed in this section. For this specific example this even leads after five iterations to an upper bound of $\ga^5 = 3.32$ and a closed-loop behavior that is almost identical to the one corresponding to the robust gain-scheduling design. Note that this in general not the case as a robust controller utilizes less information than a robust gain-scheduling controller. 
	

	\begin{figure}
		\begin{center}
			\includegraphics[width=0.6\textwidth]{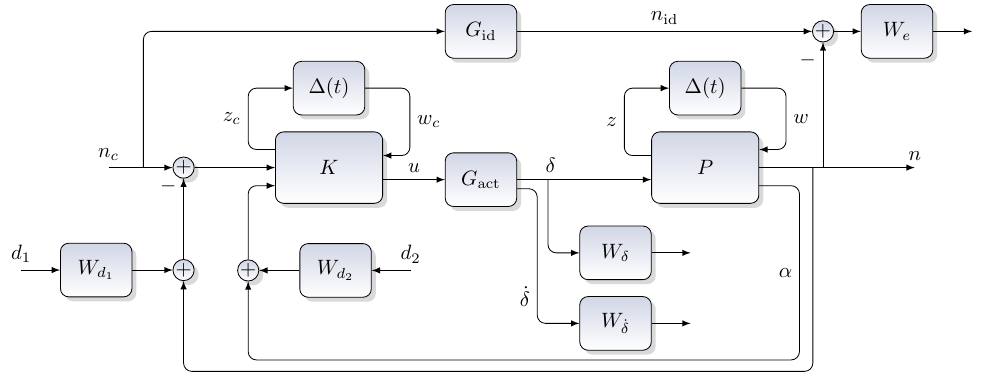}
		\end{center}
		\caption{Interconnection structure for gain-scheduled synthesis.}
		\label{RS::fig::missile_exa}
	\end{figure}

	\begin{figure}
		\begin{center}
			\begin{minipage}{0.48\textwidth}
				\begin{center}
					\includegraphics[width=\textwidth, height=0.5\textwidth, trim = 05 95 30 20, clip]{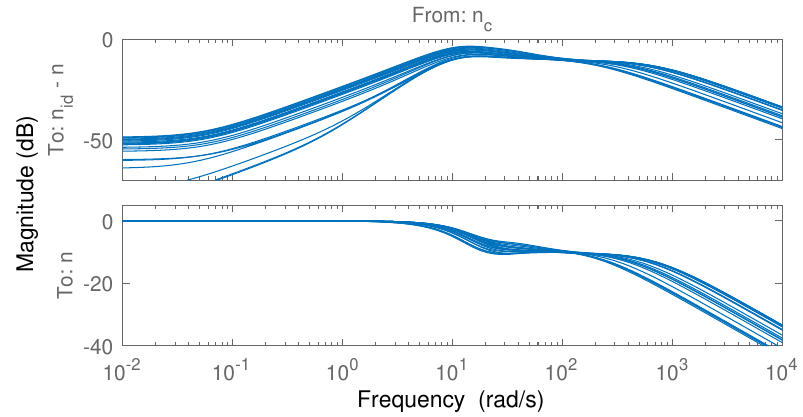}
				\end{center}
			\end{minipage}
			\hspace{2ex}
			\begin{minipage}{0.48\textwidth}
				\begin{center}
					\includegraphics[width=\textwidth, height=0.5\textwidth, trim = 05 95 30 20, clip]{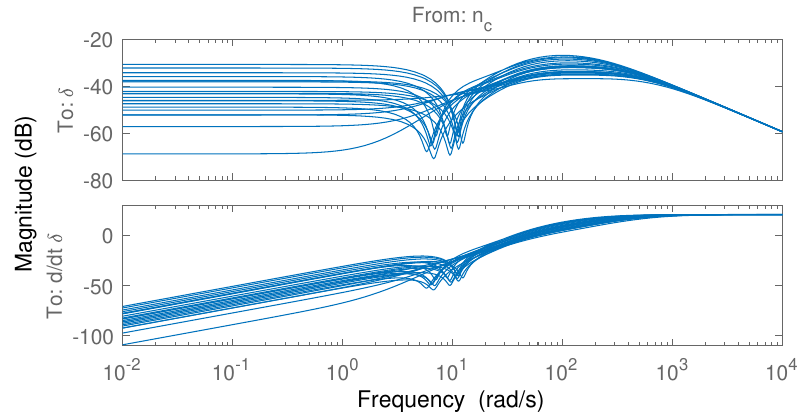}
				\end{center}
			\end{minipage}
			\begin{minipage}{0.48\textwidth}
				\begin{center}
					\includegraphics[width=\textwidth, height=0.5\textwidth, trim = 05 95 30 20, clip]{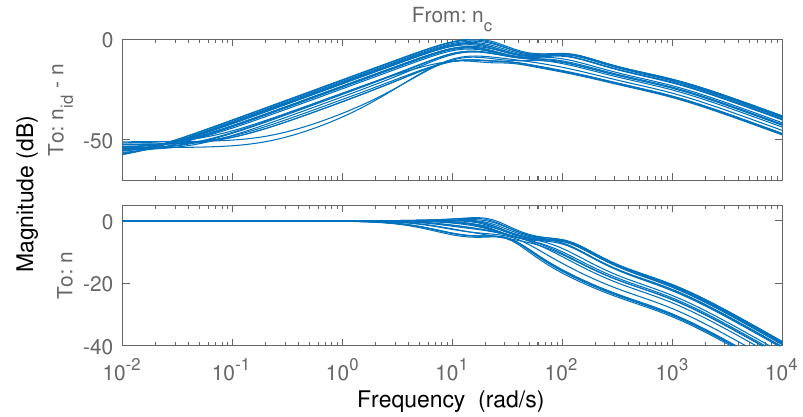}
				\end{center}
			\end{minipage}
			\hspace{2ex}
			\begin{minipage}{0.48\textwidth}
				\begin{center}
					\includegraphics[width=\textwidth, height=0.5\textwidth, trim = 05 95 30 20, clip]{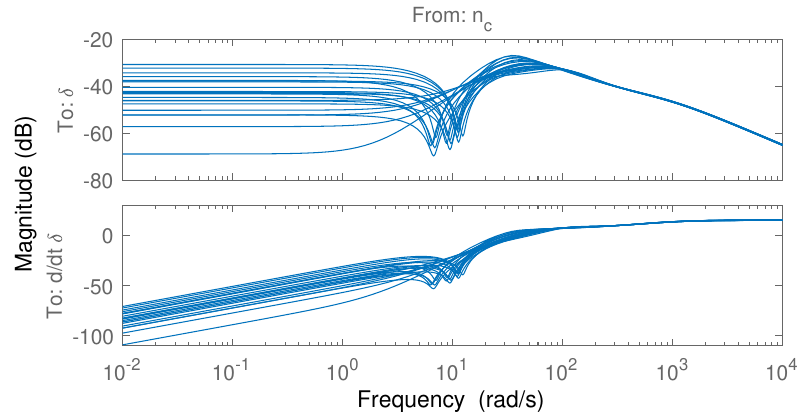}
				\end{center}
			\end{minipage}
		\end{center}
		\vspace{-2ex}
		\caption{Bode plots of the unweighted closed-loop interconnection for frozen values of $\Del$ with a gain-scheduling controller resulting from Theorem\,\ref{RS::theo::gs} (upper row) and a robust gain-scheduling controller resulting from the dual iteration (lower row).}
		\label{RS::fig::BodeGS}
	\end{figure}

	\begin{figure}
		\begin{center}
			\begin{minipage}{0.48\textwidth}
				\begin{center}
					\includegraphics[width=\textwidth, height=0.3\textwidth, trim = 15 185 30 5, clip]{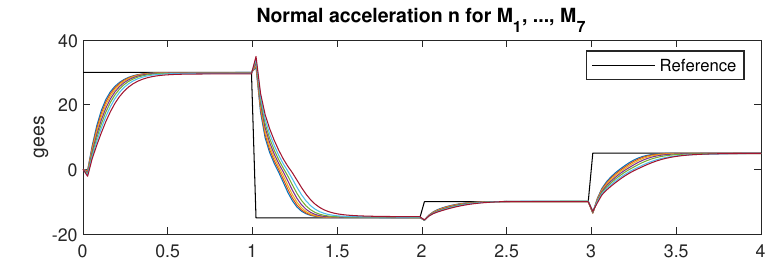}
					
					\vspace{1ex}
					\includegraphics[width=\textwidth, height=0.3\textwidth, trim = 15 170 30 5, clip]{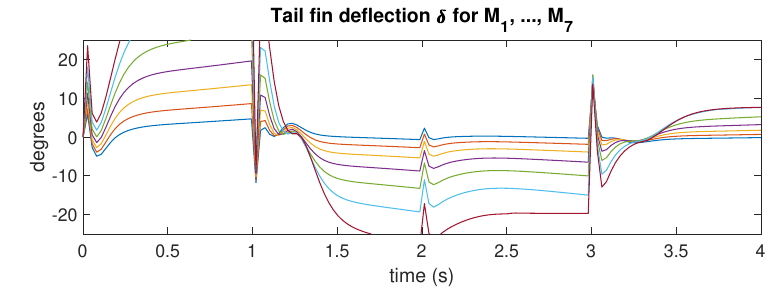}
				\end{center}
			\end{minipage}
			\hspace{2ex}
			\begin{minipage}{0.48\textwidth}
				\begin{center}
					\includegraphics[width=\textwidth, height=0.3\textwidth, trim = 15 185 30 5, clip]{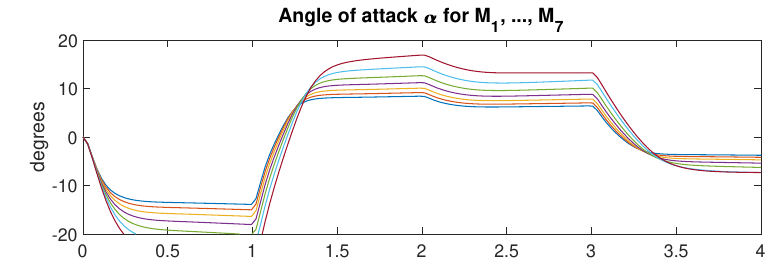}
					
					\vspace{1ex}
					\includegraphics[width=\textwidth, height=0.3\textwidth, trim = 15 170 30 5, clip]{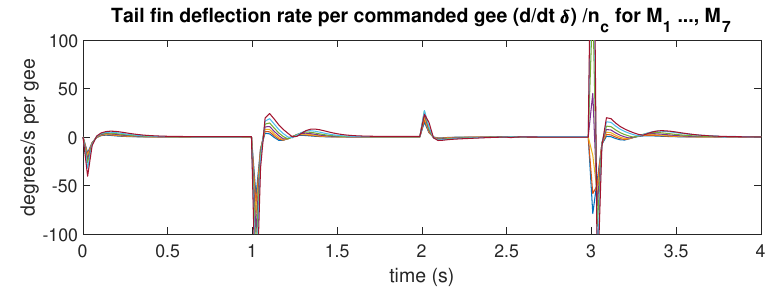}
				\end{center}
			\end{minipage}
			\begin{minipage}{0.48\textwidth}
				\begin{center}
					\includegraphics[width=\textwidth, height=0.3\textwidth, trim = 15 185 30 5, clip]{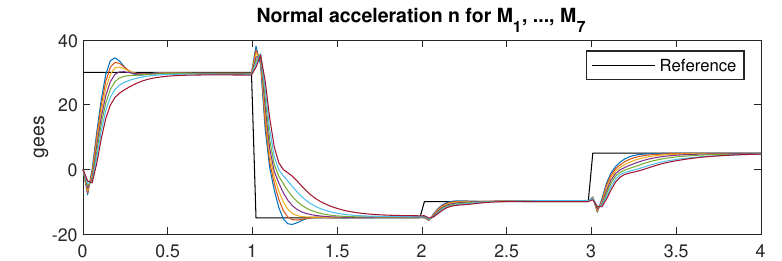}
					
					\vspace{1ex}
					\includegraphics[width=\textwidth, height=0.3\textwidth, trim = 15 170 30 5, clip]{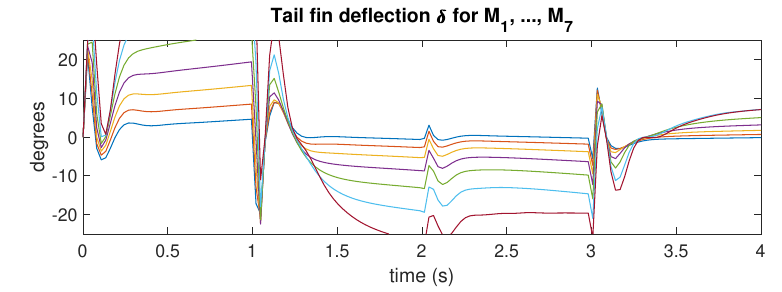}
				\end{center}
			\end{minipage}
			\hspace{2ex}
			\begin{minipage}{0.48\textwidth}
				\begin{center}
					\includegraphics[width=\textwidth, height=0.3\textwidth, trim = 15 185 30 5, clip]{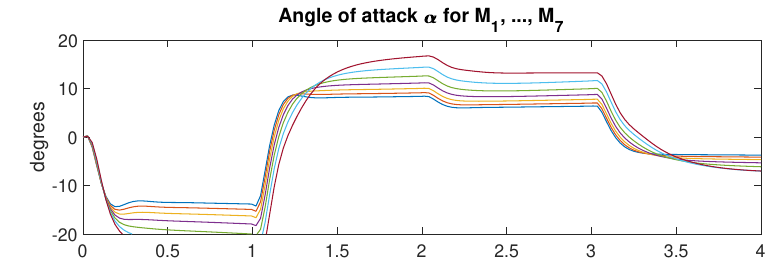}
					
					\vspace{1ex}
					\includegraphics[width=\textwidth, height=0.3\textwidth, trim = 15 170 30 5, clip]{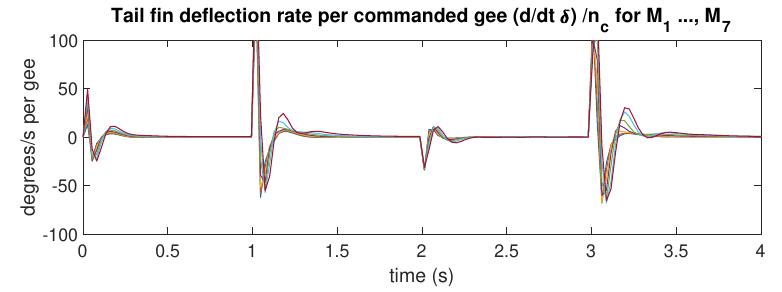}
				\end{center}
			\end{minipage}
		\end{center}
		\vspace{-2ex}
		\caption{Closed-loop trajectories for Mach numbers $M_1, \dots, M_7$ as in \eqref{RS::eq::exa::Machnumbers} and for the gain-scheduling controller (first two rows) as well as the robust gain-scheduling controller (last two rows).}
		\label{RS::fig::MISGS}
	\end{figure}

	\section{Conclusions}
	
	We demonstrate that the dual iteration, together with linear fractional representation framework, is a powerful and flexible tool to tackle various challenging and interesting non-convex controller synthesis problems especially if compared to other heuristic approaches such as the classical D-K iteration. The iteration, as introduced in \cite{Iwa97} for the design of stabilizing static output-feedback controllers, heavily relies on the elimination lemma. We extend those ideas to the synthesis of static $H_\infty$ and robust $H_\infty$ output-feedback controllers in a common fashion. As the icing on the cake, we demonstrate in terms of a missile autopilot design example, that a seamless extension to robust gain-scheduling output-feedback $H_\infty$-design is possible as well.
	
	Since the underlying elimination lemma is not applicable for numerous non-convex design problems, such as multi-objective controller design, we also provide a novel alternative interpretation of the individual steps of the dual iteration. We demonstrate that the latter interpretation allows for the extension of the dual iteration for such situations as well. 
	
	Future research could be devoted to extensions of the dual iteration to robust output-feedback design based on more elaborate analysis results. Precisely, analysis results based on parameter-dependent Lyapunov functions or on integral quadratic constraints with dynamic multipliers. It would also be very interesting and fruitful to extend the iteration for static or robust output-feedback design for hybrid and switched systems.

	\section*{Acknowledgments}
	
	
	Funded by Deutsche Forschungsgemeinschaft (DFG, German Research Foundation) under Germany's Excellence Strategy - EXC 2075 – 390740016. We acknowledge the support by the Stuttgart Center for Simulation Science (SimTech).
	
	\bibliography{literatur}%

	\appendix
	
	\section{Dualization and Elimination}

	The following technical results are highly useful for controller design purposes.
	
	\begin{lemma}
		\label{RS::lem::dualization}
		\cite{Sch01, SchWei00} Let $A\in \R^{(p+q)\times q}$, $B\in \R^{(p + q)\times p}$, $P=P^T\in \R^{(p+q)\times (p+q)}$ and suppose that $(A, B)$ and $P$ are nonsingular. Further, let $U$ and $V$ be basis matrices of $\ker(A^T)$ and $\ker(B^T)$, respectively. Then
		\begin{equation*}
			A^T P A \cl 0
			\text{ ~and~ }
			B^T P B \cge 0
			\teq{ are equivalent to }
			U^T P^{-1}U \cg 0
			\text{ ~and~ }
			V^T P^{-1}V \cle 0.
		\end{equation*}
	\end{lemma}
	
	This lemma is usually referred to as \emph{dualization lemma} and most typically applied in the case that $A = \smat{I_p \\ W}$ and $B = \smat{0 \\ I_q}$ for some matrix $W \in \R^{q \times p}$. For any nonsingular symmetric matrix $P$, Lemma \ref{RS::lem::dualization} states in this case that
	\begin{equation*}
		\mat{c}{I_p \\ W}^{\!T}\! P \mat{c}{I_p \\ W} \cl 0
		\text{ and }\mat{c}{0 \\ I_q}^{\!T}\! P \mat{c}{0 \\ I_q} \cge 0
		\teq{ are equivalent to }
		\mat{c}{-W^T \\ I_q}^{\!T}\! P^{-1} \mat{c}{-W^T \\ I_q} \cg 0
		\text{ and }
		\mat{c}{I_p \\ 0}^{\!T}\! P^{-1}\mat{c}{I_p \\ 0} \cle 0.
	\end{equation*}

	The following \emph{elimination lemma} is a very powerful tool to turn several apparently non-convex controller design problems into convex LMI feasibility problems. 
	
	\begin{lemma}
		\label{RS::lem::elimination}
		\cite{Hel99} Let $U\in \R^{r \times q}$, $V\in \R^{s \times p}$, $W\in \R^{q \times p}$, $P=P^T \in \R^{(p+q)\times(p+q)}$ and suppose that $P$ is nonsingular with exactly $p$ negative eigenvalues. Further, let $U_\perp$ and $V_\perp$ be basis matrices of $\ker(U)$ and $\ker(V)$, respectively.
		Then there exists a matrix $Z\in \R^{r \times s}$ satisfying
		\begin{equation}
			\mat{c}{I_p \\ U^TZV + W}^T P \mat{c}{I_p \\ U^TZV + W} \cl 0
			\label{RS::lem::eq::elim}
		\end{equation}
		if and only if
		\begin{subequations}
			\label{RS::lem::eq::eli}
			\begin{equation}
				V_\perp^T\mat{c}{I_p \\ W}^T P \mat{c}{I_p \\ W}V_\perp \cl 0
				\text{ ~~and~~ }
				U_\perp^T\mat{c}{-W^T \\ I_q}^T P^{-1} \mat{c}{-W^T \\ I_q}U_\perp \cg 0.
				\dlabel{RS::lem::eq::elia}{RS::lem::eq::elib}
			\end{equation}
		\end{subequations}
	\end{lemma}

	By considering the special case $P = \smat{Q & I \\ I & 0}$ and $W = 0$ for some symmetric matrix $Q$ we recover the more common variant introduced in \cite{GahApk94}. We give here a full proof of Lemma \ref{RS::lem::elimination} as it provides a scheme for constructing a solution $Z \in \R^{r \times s}$ if it exists.

	\newcommand{\inn}{\mathrm{in}_-}
	\newcommand{\im}{\mathrm{im}}
	
	\begin{proof}
		``Only if'': Multiplying \eqref{RS::lem::eq::elim} with $V_\perp$ from the right and its transpose from the left leads immediately to \eqref{RS::lem::eq::elia}. By \eqref{RS::lem::eq::elim} and since $P$ is nonsingular with exactly $p$ negative eigenvalues, we also find a matrix $B$ such that $(A, B)$ is nonsingular for $A := \smat{I_p \\ U^TZV + W}$ and such that $B^TPB \cge 0$. Applying the dualization lemma \ref{RS::lem::dualization} yields then
		\begin{equation*}
			\mat{c}{-(U^TZV + W)^T \\ I_q}^T P^{-1} \mat{c}{-(U^TZV + W)^T \\ I_q} \cg 0
		\end{equation*}
		and hence \eqref{RS::lem::eq::elib} by multiplying $U_\perp$ from the right and its transpose from the left.
		
		``If'': By the singular value decomposition we can find orthogonal $W_u$, $W_v$ and nonsingular $T_u$, $T_v$ such that
		\begin{equation*}
			U = T_u\h U W_u^T
			\teq{ and }
			V = T_v\h V W_v^T
			\teq{ with }
			%
			\h U = \diag\big(I_{q_1},~ 0_{\bullet \times q_2}\big)
			\teq{ and }
			%
			\h V = \diag\big(I_{p_1},~ 0_{\bullet \times p_2}\big).
		\end{equation*}
		With this decomposition we can express $U_\perp$ and $V_\perp$ as $W_u \big(0, I_{q_2}\big)^TX_u$ and $W_v\big(0, I_{p_2}\big)^T X_v$, respectively, for some nonsingular matrices $X_u$ and $ X_v$.
		Let us now transform the remaining matrices as $\hat P:= (\bullet)^T P \diag(W_v, W_u)$, $\h W = W_u^TWW_v$ and $\h Z := T_u^TZT_v$ with a to-be-determined matrix $Z$. 
		Further, we define 
		\begin{equation*}
			R := \left(\mat{c}{I_p \\ \h W} \mat{c}{I_{p_1} \\ 0},~ \mat{c}{0_{p \times q_1} \\ I_{q_1} \\ 0_{q_2 \times q_1}} \right) \in \R^{(p+q) \times (p_1+q_1)},\quad
			S := \mat{c}{I_p \\ \h W} \mat{c}{0\\ I_{p_2}} \in \R^{(p+q)\times p_2}
			\text{ ~and~ }
			T := \mat{c}{-\h W^T \\ I_q} \mat{c}{0 \\ I_{q_2}} \in \R^{(p+q)\times q_2}.
		\end{equation*}
		Then \eqref{RS::lem::eq::eli} is equivalent to $S^T \hat P S \cl 0$ and $T^T \hat P^{-1} T \cg 0$
		and we have
		\begin{equation*}
			\mat{c}{I_p \\ \hat U^T \hat Z \hat V + \hat W} =
			\mat{cc}{R\mat{c}{I_{p_1} \\ \h Z_{11}} & S}
			\teq{ for }
			\h Z_{11} := \mat{c}{I_{q_1} \\ 0}^T \h Z \mat{c}{I_{p_1} \\ 0}.
		\end{equation*}
		Moreover, \eqref{RS::lem::eq::elim} holds if and only if
		\begin{equation*}
			0 \cg (\bullet)^T \hat P \mat{c}{I_p
				\\ \hat U^T \hat Z \hat V + \hat W} 
			= \mat{cc}{(\bullet)^T\hat P R \mat{c}{I_{p_1} \\ \hat Z_{11}} & (\bullet)^T \\
				S^T\hat P R \mat{c}{I_{p_1} \\ \hat Z_{11}} & S^T \hat P S}.
		\end{equation*}
		Due to $S^T\h P S \cl 0$ and the Schur complement, the last inequality is equivalent to
		\begin{equation}
			\label{RS::lem::eq::elimin_intermed}
			\mat{c}{I_{p_1} \\ \hat Z_{11}}^T \left(R^T\hat PR - R^T\hat
			PS(S^T\hat P S)^{-1}S^T\hat PR\right)\mat{c}{I_{p_1}\\ \hat Z_{11}} \cl 0.
		\end{equation}
		Let $\t P$ now be the inner matrix in \eqref{RS::lem::eq::elimin_intermed} and let $\inn(M)$ denote the number of negative eigenvalues of any symmetric matrix $M$. 
		
		Next we show that $\inn(\t P) \geq p_1$. If this is true, there exists
		$\smat{Z_1 \\ Z_2} = (v_1, \dots, v_{p_1})\in \R^{(p_1 +\bullet) \times p_1}$ with $(\bullet)^T\t P \smat{Z_1 \\ Z_2} \cl 0$. We can, e.g., choose $v_1, \dots, v_{p_1}$ as the orthonormal eigenvectors corresponding to the $p_1$ negative eigenvalues of $\t P$. Via a small perturbation of $Z_1$ if necessary we can ensure that $Z_1$ is nonsingular and that $(\bullet)^T\t P \smat{Z_1 \\ Z_2} \cl 0$ remains valid. Then \eqref{RS::lem::eq::elimin_intermed} holds for $\h Z_{11} = Z_2 Z_1^{-1}$ and $Z := T_u^{-T} \smat{\h Z_{11} & \bullet \\ \bullet & \bullet} T_v^{-1}$ is a solution of \eqref{RS::lem::eq::elim} for any choice of the $\bullet$ matrices.
		
		Applying the Schur complement again yields
		\begin{equation*}
			\inn(\t P) 
			= \inn\mat{cc}{R^T\hat PR &R^T\hat P S \\ S^T\hat PR &
				S^T\hat P S} - \inn(S^T\hat PS) 
			= \inn\left((R~S)^T\hat P (R~S)\right) -p_2 
			= \inn\left(Q^T\hat P Q\right) -p_2 
		\end{equation*}
		for $Q := (R, S)$. Next, observe that $(T, Q)$ is nonsingular and, hence, we can find orthogonal $\t T$ and $\t Q$ with $\im(\t T) = \im(T)$ as well as $\im(\t Q) = \im(Q)$. With those matrices let us abbreviate $\smat{A & B \\ B^T & D} := (\t T, \t Q)^T \h P^{-1}(\t T, \t Q)$ and recall that we then have $A = \t T^T\h P^{-1}\t T \cg 0$ and 
		\begin{equation*}
			\mat{cc}{A & B \\ B^T &D}^{-1} 
			= \mat{cc}{I & 0 \\ -B^TA^{-1}&
				I}^T\mat{cc}{A & 0 \\ 0 & C - B^TA^{-1}B}^{-1}\mat{cc}{I & 0 \\
				-B^TA^{-1} & I}.
		\end{equation*}
		Then we can conclude
		\begin{equation*}
			\arraycolsep=1pt
			p = \inn(\hat P) = \inn(\h P^{-1}) = \inn(A) + \inn(C - B^TAB)
			=  \inn((C - B^TAB)^{-1}) 
			= \inn\left(\mat{c}{0 \\ I}^T\mat{cc}{A & B \\ B^T &
				C}^{-1}\!\mat{c}{0 \\ I}\right) 
			=\inn(\t Q^T\hat P\t Q).
		\end{equation*}
		Thus we finally have $\inn(\t P) = \inn\left(\t Q^T \hat
		P\t Q\right) - p_2 = \inn\left(Q^T \hat
		P Q\right) - p_2 = p - p_2 = p_1$.
	\end{proof}
	
	%

\end{document}